\documentclass[letterpaper,fontsize=11pt,DIV=10]{scrartcl}

\usepackage{amsmath,amsthm,amssymb}
\usepackage[mathscr]{eucal}
\usepackage{mathtools}
\usepackage{etoolbox}
\usepackage{caption,subcaption}
\usepackage[utf8]{inputenc}
\usepackage[english]{babel}
\usepackage{microtype}
\usepackage{lmodern}

\numberwithin{equation}{section}
\theoremstyle{plain}
\newtheorem{theorem}{Theorem}
\newtheorem{lemma}{Lemma}
\newtheorem{corollary}{Corollary}
\newtheorem{proposition}{Proposition}
\newtheorem*{theorem*}{Theorem}
\newtheorem*{lemma*}{Lemma}
\newtheorem*{corollary*}{Corollary}
\newtheorem*{proposition*}{Proposition}

\theoremstyle{definition}
\newtheorem{example}{Example}
\newtheorem{definition}{Definition}

\usepackage{url}
\usepackage[colorlinks,citecolor=blue,urlcolor=blue,linkcolor=blue,linktocpage=true]{hyperref}

\usepackage{cleveref}

\crefname{assumption}{\textbf{assumption}}{assumptions}
\Crefname{assumption}{Assumption}{Assumptions}
\crefformat{equation}{(#2#1#3)}
\crefrangeformat{equation}{(#3#1#4) to~(#5#2#6)}
\crefformat{equation}{(#2#1#3)}
\crefrangeformat{equation}{(#3#1#4) to~(#5#2#6)}
\crefname{equation}{}{}
\Crefname{equation}{}{}

\DeclarePairedDelimiter\evaluated{.}{\rvert}
\reDeclarePairedDelimiterInnerWrapper\evaluated{nostarnonscaled}{%
    \mathopen{}#2\mathclose{#3}%
}
\reDeclarePairedDelimiterInnerWrapper\evaluated{star}{%
    \mathopen{}\mathclose\bgroup #1\hskip -\nulldelimiterspace \relax
    #2\aftergroup\egroup #3%
}

\newcommand{\ConvexIndicator}[1]{\iota_{#1}}

\DeclarePairedDelimiterX\event[1]{\{}{\}}{#1}

\renewcommand{\Pr}{\mathbb{P}}

\DeclareMathOperator{\trace}{trace}

\DeclareMathOperator{\conv}{conv}

\DeclareMathOperator{\effdom}{dom}

\newcommand{\argmin}{\operatornamewithlimits{arg\,min}}

\DeclareMathOperator{\E}{\mathbb{E}}

\DeclareMathOperator{\diag}{diag}

\DeclarePairedDelimiterX\innerp[2]{\langle}{\rangle}{#1,#2}

\providecommand\given{\,|\,}
\newcommand\SetSymbol[1][]{%
  \nonscript\:#1\vert
  \allowbreak
  \nonscript\:
  \mathopen{}}

\DeclarePairedDelimiterX\set[1]\{\}{%
  \renewcommand\given{\SetSymbol[\delimsize]}
  #1 
}

\DeclarePairedDelimiterX\parens[1]\lparen\rparen{%
  \renewcommand\given{\SetSymbol[\delimsize]}
  #1   
}

\DeclarePairedDelimiterX\bracks[1]\lbrack\rbrack{%
  \renewcommand\given{\SetSymbol[\delimsize]}
  #1   
}

\DeclarePairedDelimiterX\braces[1]\lbrace\rbrace{%
  \renewcommand\given{\SetSymbol[\delimsize]}
  #1   
}

\DeclarePairedDelimiter\abs{\lvert}{\rvert}
\DeclarePairedDelimiterX\norm[1]\lVert\rVert{\ifblank{#1}{\:\cdot\:}{#1}}

\makeatletter
\newcommand{\opnorm}{\@ifstar\@opnorms\@opnorm}
\newcommand{\@opnorms}[1]{%
  \left|\mkern-1.5mu\left|\mkern-1.5mu\left|
   #1
  \right|\mkern-1.5mu\right|\mkern-1.5mu\right|
}
\newcommand{\@opnorm}[2][]{%
  \mathopen{#1|\mkern-1.5mu#1|\mkern-1.5mu#1|}
  #2
  \mathclose{#1|\mkern-1.5mu#1|\mkern-1.5mu#1|}
}
\makeatother

\DeclareMathOperator{\sign}{sign}

\newcommand{\Real}{\mathbb{R}}

\newcommand{\id}[1][]{%
  \ifstrempty{#1}{I}{I_{#1}}%
}

\newcommand{\one}[1][]{%
  \ifblank{#1}{\mathbf{1}}{\mathbf{1}_{#1}}%
}

\newcommand{\Cut}{\mathrm{Cut}}
\newcommand{\Sym}{\mathrm{Sym}}

\newcommand{\X}{\mathcal{X}}
\newcommand{\Orthogonal}{\mathcal{O}}
\newcommand{\Group}{\mathcal{G}}

\DeclareMathOperator{\SLC}{SLC}
\DeclareMathOperator{\SLT}{SLT}

\crefname{definition}{\textbf{definition}}{definitions}
\Crefname{definition}{Definition}{Definitions}
\crefname{assumption}{\textbf{assumption}}{assumptions}
\Crefname{assumption}{Assumption}{Assumptions}
\crefformat{equation}{(#2#1#3)}
\crefrangeformat{equation}{(#3#1#4) to~(#5#2#6)}

\usepackage[
  backend=biber,
  natbib=true,
  style=authoryear,
  maxbibnames=99,
  giveninits=true,
  sortcites=true,
  uniquename=false
]{biblatex}

\begin{document}

\title{Group Invariance and Computational Sufficiency}
\author{Vincent Q. Vu\\Department of Statistics\\The Ohio State University}

\maketitle

\begin{abstract}
Statistical sufficiency formalizes the notion of data reduction. In the decision theoretic interpretation, once a model is chosen all inferences should be based on a sufficient statistic.  However, suppose we start with a set of procedures 
rather than a specific model. Is it possible to reduce the data and yet still be able to compute all of the procedures?  
In other words, what functions of the data contain all of the information sufficient for computing these procedures?  This article presents some progress towards a theory of ``computational sufficiency'' and shows that strong reductions can be made for large classes of penalized $M$-estimators by exploiting hidden symmetries in the underlying optimization problems.  These reductions can (1) reveal hidden connections between seemingly disparate methods, (2) enable efficient computation, (3) give a different perspective on understanding procedures in a model-free 
setting.  As a main example, the theory provides a surprising answer to the following question: ``What do the Graphical Lasso, sparse PCA, single-linkage clustering, and L1 penalized Ising model selection all have in common?'' \end{abstract}

\section{Introduction}
\label{sec:introduction}
The extraction of information and the reduction of data are central concerns 
of statistics. One formalization of these notions is the concept of 
statistical sufficiency introduced by \citet{fisher:on} in his seminal 
article ``On the Mathematical Foundations of Theoretical Statistics'':
\begin{quote}
``A statistic satisfies the criterion of sufficiency when no other statistic 
which can be calculated from the same sample provides any additional 
information as to the value of the parameter to be estimated.''
\end{quote}
Implicit in Fisher's definition is the specification of a statistical model  
and the sense in which a sufficient statistic ``contains all of the 
information in the sample.'' In the decision theoretic interpretation, once 
a model is specified all inferences should (or might as well) be based on a 
sufficient statistic---for any procedure based on the data there is an 
equivalent randomized procedure based on a sufficient statistic 
\citep[see, e.g.,][Section 10]{halmos.savage:application}. However, 
actual data analysis does not always begin with the specification of a 
model, and it may not even make explicit use of a statistical model. 
\citet{breiman:statistical} famously described two cultural perspectives 
on data analysis:
\begin{quote}
``One assumes that the data are generated by a given stochastic data model. 
The other uses algorithmic models and treats the data mechanism as unknown.''
\end{quote}
In the former case, the statistical model gives context to ``information'' 
and statistical sufficiency can be seen as a criterion for separating the 
``relevant information'' from the ``irrelevant information.'' In the latter 
case, a statistical model is absent and statistical sufficiency is of no 
use. Rather than positing a collection of probability distributions (a statistical model), the 
data analyst might instead consider a collection of procedures or algorithms. 
So how should we formalize data reduction and what is the proper context for 
defining ``relevant information'' from an algorithmic perspective? This 
article proposes a concept called \emph{computational sufficiency}.

Computational sufficiency defines information in the context of a collection 
of procedures that share a common input domain. It is motivated in part by 
the following questions.
\begin{enumerate}
    \item Are there hidden commonalities between the procedures?
    \item Are there parts of the data that are irrelevant to all of the procedures?
    \item Can we reduce the data by removing the irrelevant parts?
    \item Can we exploit this reduction for computation?
    \item What is the most relevant core of the data?
\end{enumerate}
Precise definitions will be given in \Cref{sec:sufficiency}, but the 
basic idea is simple: a statistic (or reduction) is computationally 
sufficient if every procedure in the collection is essentially a function of 
the statistic. The data itself is computationally sufficient, because every 
procedure is already a function of the data. So the definition is only 
really useful if there are nontrivial reductions. The main point of this 
article is to show that nontrivial reductions do exist for large classes of 
procedures, and that by studying reductions within the framework of 
computational sufficiency, interesting and insightful answers can be made to 
the above questions. This provides a different perspective on understanding 
data analysis procedures when a statistical model may not be present.

The article proceeds in a manner roughly paralleling the author's own 
process of discovery. \Cref{sec:example} presents the main motivating 
example, where a commonality between three seemingly disparate methods is 
demonstrated empirically on a real dataset.  The connection between two of 
those is already known and discussed in \Cref{sec:exact_thresholding}, but 
what is surprising (at least to me) is that the 
phenomenon generalizes to a large classes of procedures. 
\Cref{sec:sufficiency} gives precise definitions for computational 
sufficiency and related concepts, and attempts to explain the parallels and 
differences with statistical sufficiency. The section also begins the main 
arc of the paper, which is a theoretical framework for the construction of 
computationally sufficient reductions.  This includes defining a class of 
procedures that generalize penalized maximum likelihood for exponential 
families (\Cref{sec:framework}). Within this class, the primary mathematical 
tool for finding commonalities is the exploitation of symmetries via group 
invariance (\Cref{sec:group_invariance}). This allows us, in 
\Cref{sec:reduction}, to construct nontrivial reductions that are 
computationally sufficient, and to return to the main example in 
\Cref{sec:single_linkage_symmetry} with deeper insight. Additional 
extensions and discussion are given in \Cref{sec:discussion}. 
\section{A motivating example}
\label{sec:example}
\begin{figure}[tbp]
\centering
\includegraphics[width=\columnwidth]{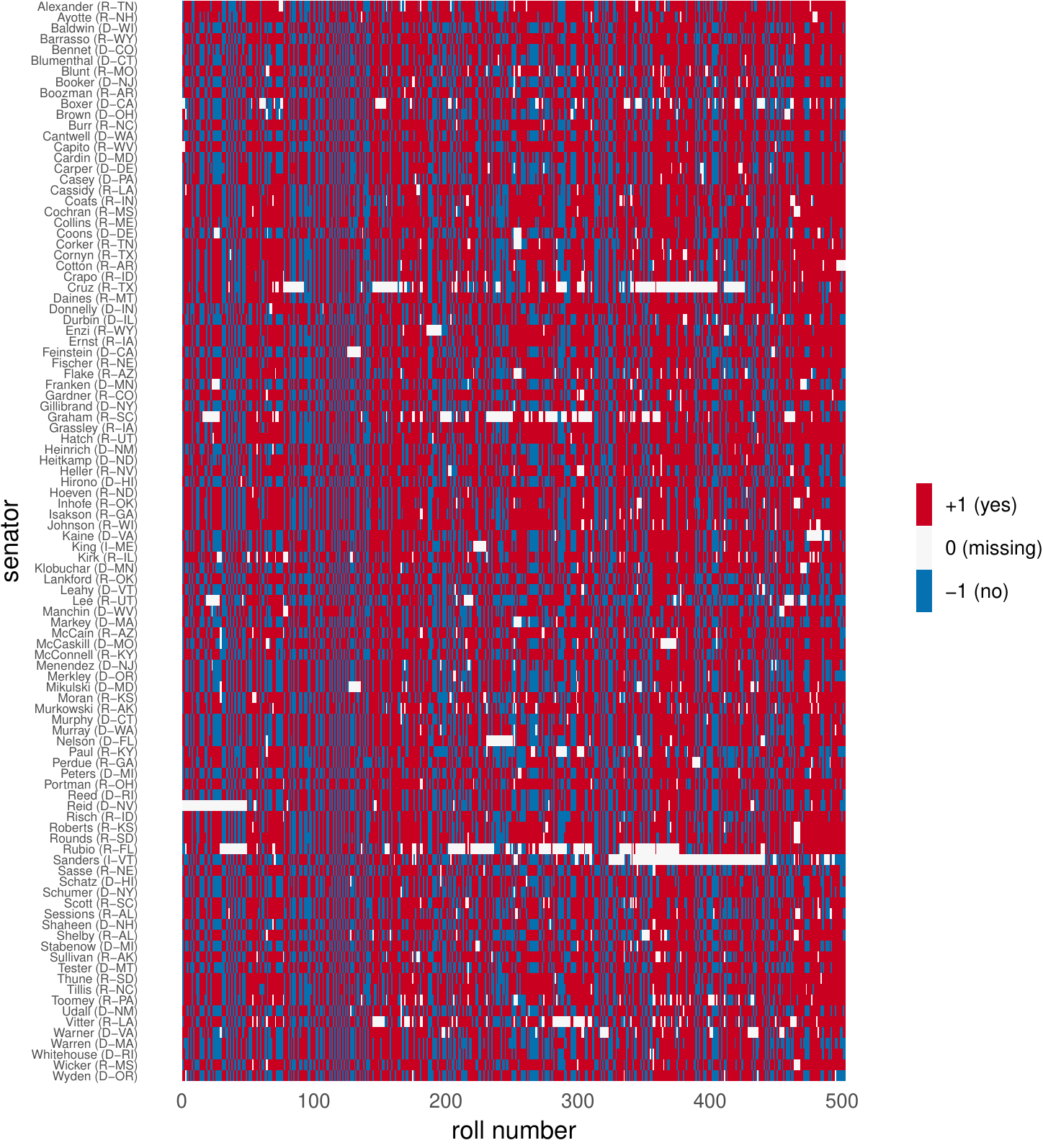}
\caption{
Senate roll call votes from the 114\textsuperscript{th} U.S. Congress. Each row is the 
voting record of a senator, arranged from earliest (left) to latest (right). 
Horizontal white stripes indicate consecutively missed votes.  The longest 
of these correspond to senators who campaigned for the 2016 U.S. 
Presidential election.
}
\label{fig:rollcall}
\end{figure}
Political polarization is a defining feature of 21\textsuperscript{st} century American
politics \citep{kohut.doherty.ea:partisan}. One manifestation of this is
in the clustering of voting patterns of political representatives in the
United States government. \Cref{fig:rollcall} displays $n=502$ senate roll
call votes from the 114\textsuperscript{th} United States Congress (January, 2015 --
January, 2017) for each of the $p=100$ senators.\footnote{The data were
collected by \citet{lewis.poole.ea:voteview} and imported with the
\texttt{Rvoteview} R package \citep{lewis:rvoteview}.} The votes are 
coded numerically as $+1$ for ``yes'', $-1$ for ``no'', and $0$ if the vote 
was missed.

\begin{figure}[t]
\centering
\begin{subfigure}[t]{0.5\columnwidth}
\centering
\includegraphics[width=0.9\columnwidth]{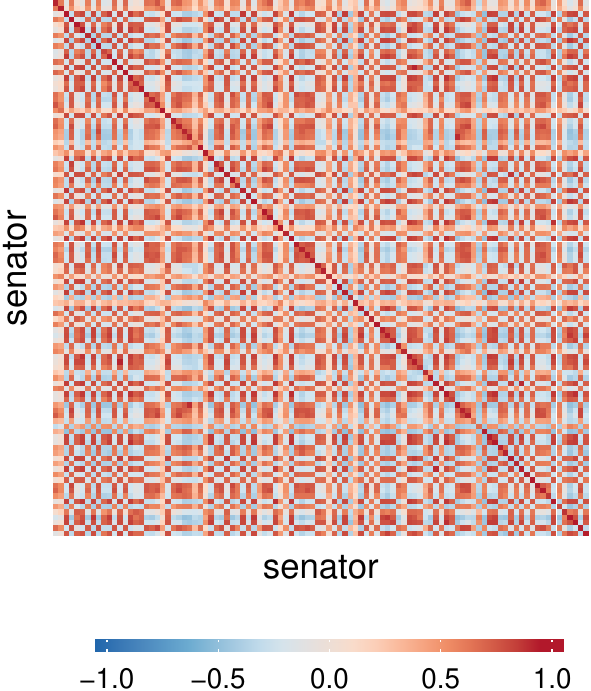}
\caption{Sorted by name}
\label{fig:sample_covariance_name}
\end{subfigure}%
~
\begin{subfigure}[t]{0.5\columnwidth}
\centering
\includegraphics[width=0.9\columnwidth]{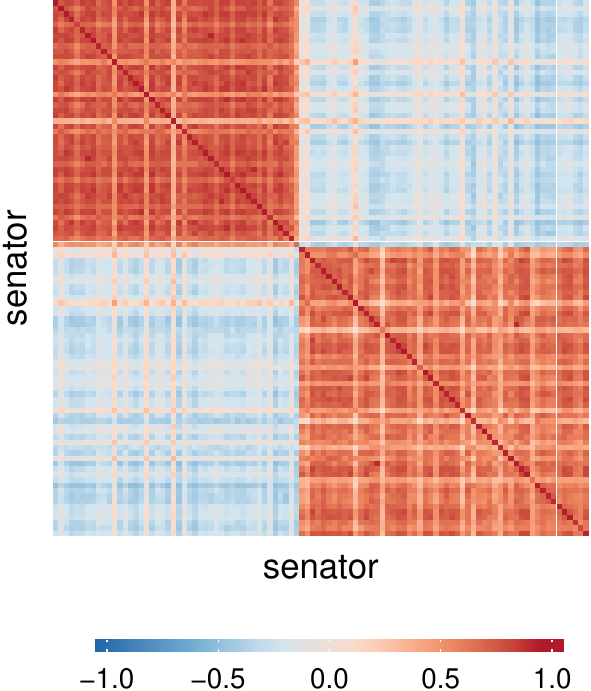}
\caption{Sorted by political party}
\label{fig:sample_covariance_party}
\end{subfigure}
\caption{
The sample covariance matrix of senators' roll call votes. Positive entries 
indicate relative agreement, while negative entries indicate relative 
disagreement. (\subref{fig:sample_covariance_name}) Senators sorted by name.
(\subref{fig:sample_covariance_party})  Senators sorted by political party 
affiliation (Democrat, Independent, or Republican) with ties broken by name. 
The two diagonal blocks correspond to Democrats (upper-left) and Republicans 
(lower-right).  Two Independent senators are placed in the middle at positions 45--46; they essentially vote with the Democrats.
}
\label{fig:sample_covariance}
\end{figure}
Relative agreement or disagreement between voting patterns of pairs of 
senators can be summarized by taking the average of the product of the 
entries of their corresponding vectors, i.e. we form a matrix 
$X \in \Real^{p \times p}$ with entries
\begin{equation}
    \label{eq:sample_covariance}
    x_{ij} 
    = \frac{1}{n} \sum_{t=1}^n v_{ti} v_{tj} 
    = \frac{(\text{\# of agreements}) - (\text{\# of disagreements})}{n}
    \,,
\end{equation}
where $v_{ti}$ is the vote of senator $i$ on roll call $t$. This can be 
viewed as an uncentered sample covariance: it is positive when the pair of 
senators tend to vote together, and it is negative when they tend to vote 
against each other. The resulting matrix $X$ is displayed in 
\Cref{fig:sample_covariance}. There is no easily discernible pattern when the 
senators are sorted alphabetically by name 
(\Cref{fig:sample_covariance_name}), 
but when sorted by political party affiliation---all but 2 senators are 
affiliated with  either the Republican Party or Democrat Party---a clear 
pattern emerges: the voting pattern of the senators appears to cluster 
according to political party (\Cref{fig:sample_covariance_party}).

\subsection{Single-linkage cluster analysis}
\label{sec:single_linkage_cluster_analysis}
\begin{figure}[tp]
\centering
\begin{subfigure}[t]{0.5\columnwidth}
\centering
\includegraphics[width=0.95\columnwidth]{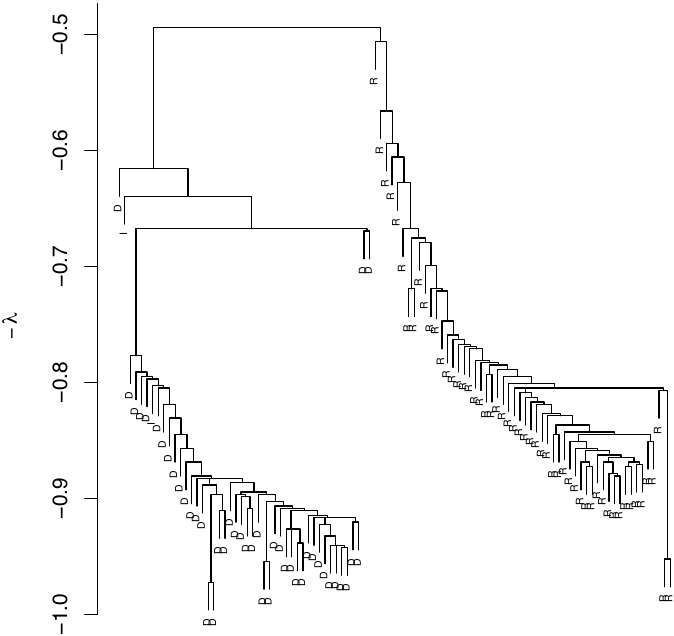}
\caption{Single-linkage dendrogram}
\label{fig:dendrogram}
\end{subfigure}%
~
\begin{subfigure}[t]{0.5\columnwidth}
\centering
\includegraphics[width=0.9\columnwidth]{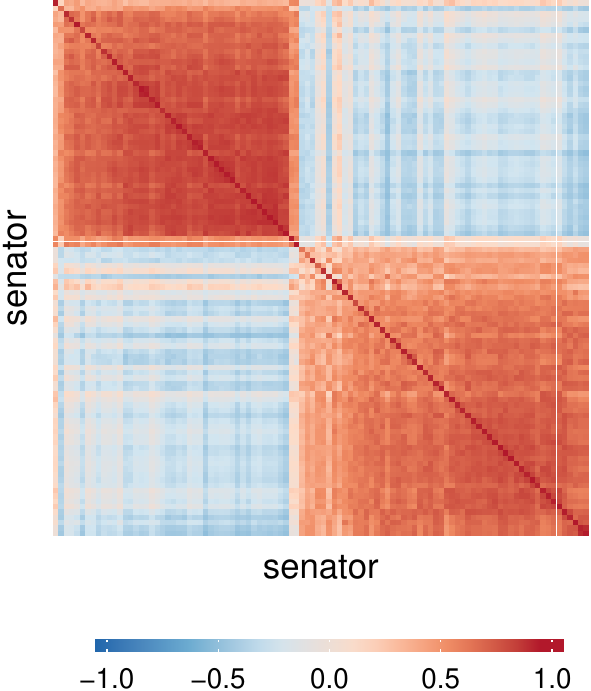}
\caption{Sample covariance matrix sorted by single-linkage}
\label{fig:sample_covariance_single_linkage}
\end{subfigure}
\caption{
Single-linkage clustering applied to the roll call vote data with
similarity measure $\abs{x_{ij}}$, i.e. magnitude of the sample covariance.
(\subref{fig:dendrogram}) The leaves of the dendrogram are labeled by
political party affiliation. Note that the value of $\lambda$ 
\emph{decreases} with increasing height in the tree. 
The main right branch of the tree consists only
of Republicans, while the left branch contains all of the  Democrats and two
Independents. (\subref{fig:sample_covariance_single_linkage}) The sample
covariance matrix is sorted according to the left-to-right ordering of leaves 
in the single-linkage dendrogram. The two diagonal blocks again correspond to
Democrats (upper-left) and Republicans (lower-right). Two independent
senators are mixed together with the Democrats.
}
\label{fig:single_linkage}
\end{figure}
To contrast with the nominal clustering provided by political party
affiliation, the data analyst might instead employ an intrinsic cluster
analysis using the roll call votes alone.  One well-known technique is
single-linkage clustering \citep{sneath:application}, which is a
hierarchical clustering procedure that takes a similarity measure as input.
The algorithm starts from the finest clustering, where each senator is placed
in his/her own cluster, and iteratively merges the most similar pairs of
clusters until a single cluster remains. The similarity between a pair of
clusters is defined to be the maximum of the pairwise similarities between 
their respective constituents, so with each merge there is always a 
``single link'' that binds the clusters together. The results of the process 
are encoded in a dendrogram: a tree whose leaves are senators and  
internal vertices are merges. The height of a vertex corresponds 
to the similarity between a pair of clusters just before merging. Cutting 
the dendrogram at different heights induces different, but hierarchically arranged, clusterings.  (See \Cref{sec:single_linkage_extras} for a 
graph-theoretic description of single-linkage.)

\Cref{fig:single_linkage} shows the result of single-linkage 
clustering applied to the roll call data with similarity measure 
$\abs{x_{ij}}$. Though the choice of absolute sample covariance may seem 
odd, the rationale for this choice will become clear later. Looking at the 
dendrogram (\Cref{fig:dendrogram}), we see that the large gap between the 
merge heights of the Democrats (in the left branch of the dendrogram) and 
the Republicans (in the right branch of the dendrogram) reflects the 
polarization of their voting patterns. Comparing the sample covariance 
matrix sorted by political party (\Cref{fig:sample_covariance_party}) and by 
single-linkage (\Cref{fig:sample_covariance_single_linkage}), we see that 
single-linkage not only recovers the party affiliation, but the relative 
smoothness of the gradients of diagonal blocks suggests that single-linkage 
may have also discovered some finer structure in the data.

\subsection{Sparse multivariate methods}
Continuing with a progression of technique, the data analyst may find 
himself enticed by more recent and potentially more powerful multivariate methods employing sparsity.  Two such methods are sparse inverse covariance estimation and sparse principal components analysis (or sparse PCA). 
\citet[Chapters 8.2 and 9]{hastie.tibshirani.ea:statistical} give an excellent overview and bibliographic notes. 
These methods can be viewed as sparse estimators of functionals of a 
population covariance matrix $\Sigma$. 
In one case, the functional is simply the inverse $\Sigma^{-1}$, 
while for PCA the functional is the projection matrix of the subspace 
spanned by the $k$ leading eigenvectors (or principal component 
directions). There are many different formulations of these methods; 
here we consider two formulations based on convex programming: Graphical Lasso~
\citep{yuan.lin:model,friedman.hastie.ea:sparse,banerjee.ghaoui.ea:model} and Sparse PCA via Fantope Projection~\citep{.el-ghaoui.ea:direct,vu.cho.ea:fantope}.

Graphical Lasso is a penalized maximum likelihood method based on the 
convex optimization problem,
\begin{equation}
    \label{eq:glasso}
    \begin{aligned}
        &\text{minimize}
          &   &
        -\log\det(\theta) + \innerp{X}{\theta}
        + \lambda \norm{\theta}_1 \,,
    \end{aligned}
\end{equation}
where $\innerp{}{}$ denotes the trace inner product, $\lambda \geq 0$ 
is a tuning parameter, and $\norm{}_1$ is the $\ell_1$ norm--the sum of the 
absolute values of the coordinates of its argument.  \Cref{eq:glasso} 
is a penalized Gaussian log-likelihood.  The $\ell_1$ penalty encourages 
sparsity in the solution, with larger values of $\lambda$ yielding solutions 
with more zero entries.  If $\theta$ is the inverse covariance matrix of a 
multivariate Gaussian distribution, then the interpretation is that 
$\theta_{ij}$ is $0$ if and only if variables $i$ and $j$ are conditionally 
independent, given the other variables.

Sparse PCA via Fantope Projection is also based on convex optimization, 
but more specifically, it is based the semidefinite optimization problem, 
\begin{equation}
    \label{eq:fps}
    \begin{aligned}
        &\text{maximize}
          &   &
          \innerp{X}{\theta}
        - \lambda \norm{\theta}_1
        \\
        &\text{subject to}
          &   & \theta \in \mathcal{F}^k
          \,,
    \end{aligned}
\end{equation}
where
\begin{equation}
    \label{eq:fantope}
    \mathcal{F}^k 
    \coloneqq 
    \set{\theta \given 0 \preceq \theta \preceq \id, \trace(\theta) = k}
    \,.
\end{equation}
This can be viewed as an $\ell_1$ penalized convex relaxation of the variance 
maximization problem. The constraint set $\mathcal{F}^k$ 
consists of symmetric matrices with 
eigenvalues between $0$ and $1$ and whose trace is equal to $k$.  
This is called the Fantope and it is the convex hull of rank-$k$ 
projection matrices \citep{vu.cho.ea:fantope}. The interpretation of  
$\lambda$ and the $\ell_1$ penalty are similar to the Graphical Lasso---they 
influence the sparsity of the solution. Sparsity of a projection implies 
that the principal components depend on a small number of variables.
There is, however, an additional user-chosen parameter $k$ that specifies 
the desired rank of estimated projection matrix and hence the number of 
principal components.

\begin{figure}[tp]
\centering
\begin{subfigure}[t]{0.33\columnwidth}
\centering
\includegraphics[width=0.9\columnwidth]{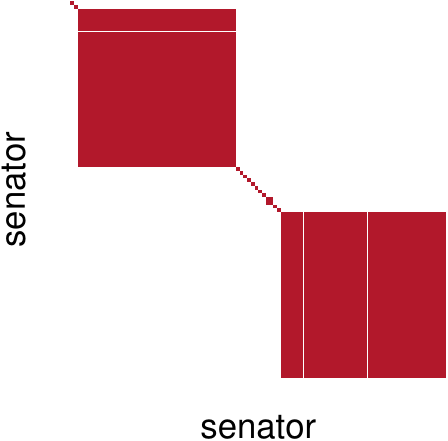}
\caption{Single-linkage clustering}
\label{fig:single_linkage_clustering}
\end{subfigure}%
~
\begin{subfigure}[t]{0.33\columnwidth}
\centering
\includegraphics[width=0.9\columnwidth]{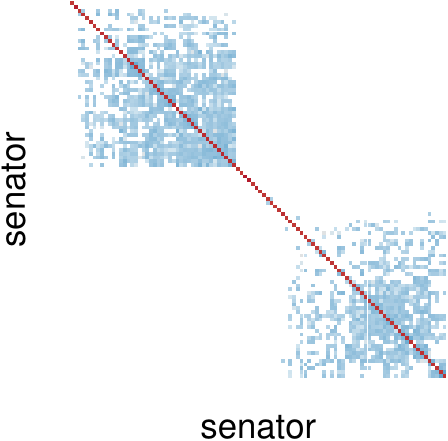}
\caption{Graphical Lasso}
\label{fig:glasso}
\end{subfigure}%
~
\begin{subfigure}[t]{0.33\columnwidth}
\centering
\includegraphics[width=0.9\columnwidth]{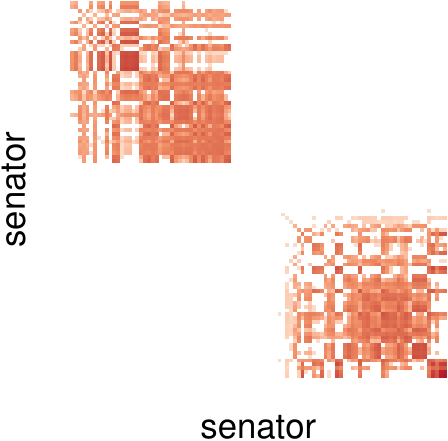}
\caption{Sparse PCA via Fantope Projection}
\label{fig:fps}
\end{subfigure}
\caption{
(\subref{fig:single_linkage_clustering}) The single-linkage clustering 
matrix and estimates of the (\subref{fig:glasso}) inverse covariance matrix 
and (\subref{fig:fps}) principal components projection matrix. 
These are all based on the senator-senator sample covariance matrix, and in 
all three cases the tuning parameters and dendrogram cut height are set 
to $\lambda = 0.7$.
}
\label{figure:comparison}
\end{figure}

\Cref{figure:comparison} shows the results of Graphical Lasso and 
Sparse PCA via Fantope Projection\footnote{Software implementations are 
provided by the R packages \texttt{glasso} 
\citep{friedman.hastie.ea:glasso} and \texttt{fps} 
\citep{vu:fantope}, respectively.} 
applied to the senator-senator sample covariance 
matrix~\Cref{eq:sample_covariance}. The tuning parameter for both procedures 
was set to $\lambda = 0.7$ and $k=5$ was chosen for Sparse PCA. \Cref{fig:single_linkage_clustering} shows 
the result of cutting single-linkage dendrogram at $\lambda = 0.7$ as 
a clustering matrix---a binary matrix with $1$ in entry $i,j$ if and only if 
$i$ and $j$ are in the same cluster. Remarkably, the block-diagonal 
structure is very similar across all three methods, and all three methods 
capture large chunks of the two major political parties.  In fact, the 
supports of both the Graphical Lasso and Sparse PCA estimates are contained 
in the support of the single-linkage clustering matrix, and this continues 
to hold for other choices of $\lambda$. One possible 
summary of this phenomenon is that the Graphical Lasso and Sparse 
PCA seem to be refinements of single-linkage.  While single-linkage easily 
discovers the two big blocks, the more sophisticated techniques reveal 
finer structure within the blocks. 
\subsection{Exact thresholding, the Graphical Lasso, and more}
\label{sec:exact_thresholding}

The similarity between Graphical Lasso and single-linkage  clustering shown in
\Cref{figure:comparison} is an instance of  the ``exact thresholding''
phenomenon first observed by
\citet{mazumder.hastie:exact,witten.friedman.ea:new}.  In brief, they
proved that the graph formed by thresholding the entries of $X$ at level
$\lambda$---by setting to zero any entry with $\abs{x_{ij}} \leq
\lambda$---and the estimated inverse covariance graph produced by the
Graphical Lasso with tuning parameter $\lambda$ have exactly the same
connected components. In other words, the thresholded matrix $X$ and the
Graphical Lasso estimate have exactly the same block-diagonal structure. 
The proofs of \citet{mazumder.hastie:exact,witten.friedman.ea:new} are 
similar; they are based on direct examination of the Karush--Kuhn--Tucker 
(KKT) optimality conditions for \Cref{eq:glasso} and 
exploit special properties of the log-determinant. 
Building on the exact thresholding phenomenon,
\citet{tan.witten.ea:cluster} later observed that the connected components
of the Graphical Lasso correspond to the clusters of single-linkage with
similarity measure $\abs{x_{ij}}$.

The connection between exact thresholding, the Graphical Lasso, and 
single-linkage clustering has several implications, and two perspectives have  
emerged in the literature: algorithmic and methodological. 
\citet{witten.friedman.ea:new,mazumder.hastie:exact} showed 
that exact thresholding leads to faster algorithms for the Graphical Lasso. For a $p \times p$ input $X$, generic solvers for the Graphical Lasso 
optimization problem \Cref{eq:glasso} have $O(p^3)$ time complexity per 
iteration.  On the other hand, thresholding and identifying the connected 
components has worst case time complexity $O(p^2)$. Once the connected 
components are identified, the parameter space, i.e. the feasible set, of 
the optimization problem can be reduced and decomposed to smaller, 
separate blocks. This reduces the Graphical Lasso optimization problem into 
separate smaller problems that can be solved in parallel and more quickly 
than the original problem. This algorithmic aspect of phenomenon has been 
extended on a case-by-case basis to various generalizations of the Graphical 
Lasso \citep{danaher.wang.ea:joint,mohan.london.ea:node-based,qiao.guo.ea:functional,tan.london.ea:learning,zhu.shen.ea:structural}. 

On the methodological side, \citet{gsell.taylor.ea:adaptive} used the 
monotonicity property implied by the exact thresholding phenomenon to 
develop adaptive sequential hypothesis tests based on examining ``knots'' 
in the Graphical Lasso solution path---these knots correspond to the merge 
events in single-linkage. \citet{tan.witten.ea:cluster} took a critical 
perspective by using the connection to motivate alternative estimators of 
the inverse covariance matrix. They noted that Graphical Lasso 
could be viewed as a two-step procedure. In the first step, it performs 
single-linkage clustering with similarity measure $\abs{x_{ij}}$. In the 
second step, it performs penalized maximum likelihood estimation 
on each connected component. Focusing on the first step, they 
argue that single-linkage clustering has an undesirable ``chaining'' 
effect \citep[see, e.g.,][]{hartigan:consistency}, and propose to 
replace it with an alternative clustering algorithm. They call 
the resulting two-step estimator ``Cluster Graphical Lasso,'' and 
demonstrate empirically some of its advantages over the Graphical Lasso.

The implications of exact thresholding discussed above add insight to our
collective understanding of the Graphical Lasso. Recalling the questions
posed in the introduction, we see that the exact thresholding phenomenon
explains that there are hidden commonalities between single-linkage
clustering and the Graphical Lasso, and that this can be exploited for
reduction in computation. Yet there is much more to the phenomenon. In
\Cref{sec:single_linkage_symmetry}, we will see that not only can the
parameter space of the Graphical Lasso problem be reduced, but that the
input $X$ to the Graphical Lasso optimization problem  can essentially be
replaced by $\SLT_\lambda(X)$, the \emph{single-linkage thresholding
operator}: 
\begin{equation*}
    \bracks{\SLT_\lambda(X)}_{ij} = 
    \begin{cases}
        x_{ij} 
        &\text{ if $i \sim_{\lambda} j$, and} \\
        0 & \text{otherwise,}
    \end{cases}
\end{equation*}
where $i \sim_{\lambda} j$ means that $i$ and $j$ are in the same single-linkage cluster at level $\leq \lambda$. 
This will demonstrate that there are irrelevant parts of the data that can  
be removed, and perhaps more surprisingly, 
\Cref{sec:reduction,sec:single_linkage_symmetry} will show that this 
type of phenomenon extends beyond the Graphical Lasso and holds 
simultaneously for many other procedures. 
\section{Computational sufficiency}
\label{sec:sufficiency}

Given a collection of procedures that share a common input domain 
$\mathcal{X}$, we would like to be able to reduce the input and yet still 
be able to compute each procedure.  So our goal is to define concepts that  
identify the information that is sufficient and necessary for 
computing all of the procedures. Some of the procedures considered in 
\Cref{sec:example} are based on optimization. Such estimators may 
not always be uniquely defined and they may not even exist for some inputs. 
For example, optimization problem 
\Cref{eq:fps} may have more than one solution, and \Cref{eq:glasso} 
may not even have a solution if $\lambda = 0$. 
With these complications in mind, we define a procedure 
to be a set-valued function on $\mathcal{X}$.
For example, let $T(x)$ denote the set of solutions of \Cref{eq:fps} 
when $X = x$. Then every element of $T(x)$ achieves the same value of the 
objective function. Defining a procedure to be a set-valued function 
provides a convenient way to describe equivalent and/or possibly 
void results.

\subsection{Definitions}
Let $\mathcal{M}$ be a collection of set-valued functions on $\mathcal{X}$. 
This is our collection of procedures.  The \emph{effective domain} 
of a set-valued function $T$ is defined to be
\begin{equation*}
    \effdom T \coloneqq \set{x \in \mathcal{X} \given T(x) \neq \emptyset } \,.
\end{equation*}
If the data analyst is content to obtain \emph{any} singleton 
from $T(x)$, whenever $x \in \effdom T$, then 
\begin{definition}
\label{def:computational_sufficiency}
A function $R$ on $\mathcal{X}$ is 
\emph{computationally sufficient} for $\mathcal{M}$ if for each  
$T \in \mathcal{M}$, there exists a set-valued function $f$ such that 
$f(T, R(x)) \neq \emptyset$ for all $x \in \effdom T$ and 
\begin{equation}
    \label{eq:computational_sufficiency_criterion}
    f(T, R(x)) \subseteq T(x)
    \quad \text{for all} \quad x \in \X
    \,.
\end{equation}
When this is the case, we may refer to $R$ as being a \emph{reduction}.
\end{definition}
Criterion \Cref{eq:computational_sufficiency_criterion} says that every 
$T \in \mathcal{M}$ is \emph{essentially} a function of $R$---up to the 
equivalence implied by the set-valuedness of $T$.  If $T$ is singleton-valued, 
then \Cref{eq:computational_sufficiency_criterion} becomes an equality. 
The identity map is trivially computationally sufficient, but 
clearly provides no reduction.
So in the pursuit of reduction without loss of information, 
there is an obvious interest in finding a maximal reduction. The following 
definitions parallel the definitions of necessary and minimal sufficient 
statistics.
\begin{definition}
\label{def:computational_necessity}
A function $U$ is \emph{computationally necessary} for $\mathcal{M}$  
if for each $R$ that is computationally sufficient for $\mathcal{M}$, 
there exists $h$ such that 
\begin{equation*}
    U(x) = h(R(x)) \quad \text{for all} \quad x \in \mathcal{X} 
    \,.
\end{equation*}
If $U$ is computationally necessary and computationally sufficient, 
then we say that $U$ is \emph{computationally minimal}.
\end{definition}
By definition, every singleton-valued $T \in \mathcal{M}$ is computationally 
necessary for $\mathcal{M}$. This simple observation leads to the following 
result and needs no proof.
\begin{lemma}
\label{lem:trivially_minimal}
If $R \in \mathcal{M}$ is singleton-valued and 
computationally sufficient for $\mathcal{M}$, then 
$R$ is computationally minimal.
\end{lemma}
This seemingly trivial statement will turn out to be a useful device for 
establishing computational minimality.  An immediate consequence is that if 
any $T \in \mathcal{M}$ is a bijection, then the identity map 
is computationally minimal and no further reduction is possible. So in order 
for a nontrivial reduction to exist, it is necessary that all of the 
procedures in $\mathcal{M}$ be noninvertible.

\subsection{Reductions and partitions}
\label{sec:partitions}
Nontrivial computationally sufficient reductions are only possible when 
the procedures under consideration are themselves nontrivial reductions. 
Heuristically, this means that the preimages of the results of different 
procedures should be large and coincide with one another. 
If every $T \in \mathcal{M}$ is singleton-valued, then 
the criterion of computational sufficiency can be expressed more simply as
\begin{equation}
    \label{eq:computational_sufficiency_singleton}
    T(x) = f(T, R(x)) \quad \text{for all} \quad x \in \effdom T \,.
\end{equation}
When this is the case, computational sufficiency can be stated in terms of 
the partitions of $\X$. For a function $h$ on $\X$, let 
\begin{equation*}
    \sigma(h) \coloneqq \bigcup_{x \in \X} 
    \set[\big]{ \set{ u \given T(u) = T(x) } }\,,
\end{equation*}
i.e. $\sigma(h)$ is the the partition of $\X$ induced by $h$. 
We can order the set of all partitions of $\X$ by refinement, writing 
$\alpha \preceq \beta$ if $\alpha$ refines $\beta$. 
Then $R$ is computationally 
sufficient for $\mathcal{M}$ if and only if 
\begin{equation*}
    \sigma(R) \preceq \sigma(T) \quad \text{for all} \quad T \in \mathcal{M}
    \,.
\end{equation*}
A function $U$ on $\X$ is computationally necessary for $\mathcal{M}$ if 
and only if
\begin{equation*}
    \sigma(R) \preceq \sigma(U)
\end{equation*}
for all computationally sufficient $R$. 
Since ordering by refinement turns the set of all partitions of $\X$ into a complete lattice, there exists a 
coarsest partition that refines all of the partitions $\sigma(T)$ 
induced by $T \in \mathcal{M}$.  That greatest lower bound is the partition 
induced by a computationally minimal reduction for $\mathcal{M}$. 
This description of computational sufficiency in terms of partitions is 
conceptually useful, but it seems practically impossible to reason about specific procedures in terms of the partitions that they induce.

\subsection{Computational sufficiency versus statistical sufficiency}
Expression \Cref{eq:computational_sufficiency_singleton} bears a strong 
resemblance to the factorization criterion of the Fisher--Neyman Theorem 
for statistical sufficiency. Suppose that $\mathcal{P}$ is a family of 
positive densities on $\mathcal{X}$. Then a statistic $R$ is sufficient 
for $\mathcal{P}$ if and only if there exist $g$ and $h$ such that for all 
$q \in \mathcal{P}$
\begin{equation*}
    q(x) = g(q, R(x)) h(x) \quad \text{for all} \quad x \in \mathcal{X} 
    \,.
\end{equation*}
Fixing any $q_0 \in \mathcal{P}$ and dividing both sides, the above 
criterion is equivalent to the existence of $f$ such that for all 
$q \in \mathcal{P}$, 
\begin{equation*}
    \frac{q(x)}{q_0(x)} = f(q/q_0, R(x)) 
    \quad \text{for all} \quad x \in \mathcal{X} 
\end{equation*}
\citep[c.f.][Corollary 2]{halmos.savage:application}. 
Letting $\mathcal{Q} = \set{q/q_0 \given q \in \mathcal{P}}$ we see 
immediately that there is a clear computational interpretation of 
statistical sufficiency: 
$R$ is statistically sufficient for $\mathcal{P}$ if and only if it is 
computationally sufficient for the likelihood ratios $q/q_0$.  The 
connection goes further.  The following result is a straightforward 
consequence of definitions.
\begin{lemma}
\label{lem:trivial_minimal_reduction}
Let $\Lambda$ be a function on $\mathcal{X}$ with values taking the 
form of a function on $\mathcal{M}$.  For each $x \in \mathcal{X}$ define
\begin{equation*}
    \parens{\Lambda(x)}(T) = T(x) 
    \quad \text{for all} \quad T \in \mathcal{M}.
\end{equation*}
Then $\Lambda$ is computationally minimal for $\mathcal{M}$.
\end{lemma}
Applying this to $\mathcal{Q}$, we see that
that the likelihood ratios are computationally minimal, and hence 
statistically sufficient. We can also deduce that they are statistically 
minimal sufficient.

One way to view the philosophical difference 
between computational sufficiency and statistical sufficiency is that 
definition of statistical sufficiency starts from conditional 
probability and is in essence about isolating the information that 
is sufficient for computing conditional expectation for any 
distribution in the model. In the measure-theoretic setting this, 
unfortunately, entails substantial technical complications that preclude 
the conclusion of the Fisher--Neyman factorization theorem from always 
being true. Computational sufficiency, on the other hand, starts from 
a definition that is analogous to the factorization criterion, and 
that directly isolates the information that is sufficient for computing 
either the procedures or the likelihood.

\section{Expofam-type estimators}
\label{sec:framework}
To demonstrate the general existence and feasibility of computationally sufficient reductions 
we introduce a framework for procedures that are generalizations of penalized 
maximum likelihood for exponential family models.  Let $\mathcal{X}$ be a 
Euclidean space equipped with an inner product $\innerp{}{}$ which induces a 
norm $\norm{}$. We say that a set-valued function $T$ on $\mathcal{X}$ 
is an \emph{expofam-type} estimator if it has the form
\begin{equation}
    \label{eq:expofam-type}
    T(x) = 
    \argmin_\theta A(\theta) - \innerp{x}{\theta} + h_C(\theta) \,,
\end{equation}
where $A : \mathcal{X} \to \Real \cup \set{+\infty}$ is the called the 
\emph{generator} of $T$ and and 
$h_C : \mathcal{X} \to \Real \cup \set{+\infty}$ is the support function,
\begin{equation*}
    h_C(\theta) = \max_{z \in C} \innerp{z}{\theta} \,,
\end{equation*}
of a nonempty, closed and convex set $C$.  We will assume that $A$ is closed 
(lower semicontinuous), convex and proper (finite for at least one value in 
$\mathcal{X}$). 
The optimization problem in \Cref{eq:expofam-type} may possibly have multiple 
or no solutions depending on $x$, so it is important that we view $T$ as being a set-valued function on $\mathcal{X}$.

There are several important features of this formulation.  The objective 
function in \Cref{eq:expofam-type} should be viewed as being the sum of two 
parts: a loss, $A(\theta) - \innerp{x}{\theta}$, and a penalty, 
$h_C(\theta)$. Both parts are closed convex functions and so their sum is 
also a closed convex function. 
The loss strictly generalizes the negative log-likelihood of an exponential 
family. We only require that $A$ be a closed, convex and proper function, 
so in general it may not be the log-partition function of an exponential 
family of distributions. The penalty generalizes seminorms, and in fact any 
closed sublinear function can be viewed as the support function of some closed 
convex set 
\citep[e.g.][Theorem 3.1.1]{HUL:fundamentals}.
The importance of viewing the penalty in this way is that it establishes a 
link between functions and sets. 

Many existing procedures fit into the framework of \Cref{eq:expofam-type}, 
and it is useful to organize them according to their generator $A$ and penalty 
support set $C$. \Cref{tab:A_functions,tab:C_sets} gives some examples. 
There are numerous others, but our main focus in this article will be on the 
examples that follow.

\begin{table}[t]
\centering
\begin{tabular}{l|c}
    \hline
    Method & $A(\theta)$ \\
    \hline
    Least squares 
    & $\frac{1}{2} \norm{\theta}^2$ \\
    Constrained least squares
    & $\frac{1}{2} \norm{\theta}^2 + \ConvexIndicator{K}(\theta)$ \\
    Inverse covariance 
    & $-\log\det(-\theta)$ \\
    PCA 
    & $\ConvexIndicator{\mathcal{F}^k}(\theta)$ \\
    Ising model 
    & $\log \sum_{u \in \set{-1,+1}^p} \exp(\innerp{u u^T}{\theta})$ \\
    \hline
\end{tabular}
\caption{Examples of generators $A$}
\label{tab:A_functions}
\end{table}

\begin{table}[t]
\centering
\begin{tabular}{l|c}
    \hline
    Penalty & $C$ \\
    \hline
    Lasso ($\ell_1$)
    & $\set{z \given \norm{z}_\infty \leq \lambda}$ \\
    Group Lasso ($\ell_{1,2}$)
    & $\set{z \given \max_i \norm{z_i} \leq \lambda}$ \\
    General norms ($\nu$)
    & $\set{z \given \nu_*(z) \leq \lambda}$ \\
    Cone constraint ($\ConvexIndicator{K}$)
    & $K^\circ = \set{z \given h_K(z) \leq 0}$ \\
    \hline
\end{tabular}
\caption{Examples of penalty support sets $C$}
\label{tab:C_sets}
\end{table}

\subsection{Penalized least squares with constraints}
\label{example:least_squares}
The most basic example is obtained by taking 
\begin{equation*}
    A(\theta) = \frac{1}{2} \norm{\theta}^2 \,,
\end{equation*}
so that \Cref{eq:expofam-type} becomes equivalent to penalized 
least squares:
\begin{equation*}
    T(x) 
    = \argmin_\theta \frac{1}{2}\norm{x - \theta}^2 + h_C(\theta) 
    \,.
\end{equation*}
Sometimes it can be useful to put constraints on $\theta$, 
say $\theta \in K$ for some closed convex set $K$. 
We can incorporate this constraint into $A$ by adding the convex indicator 
function 
\begin{equation*}
    \ConvexIndicator{K}(\theta) 
    = 
    \begin{cases}
        0 &\text{if $\theta \in K$, and} \\
        +\infty &\text{otherwise.}
    \end{cases}
\end{equation*}
Then with 
\begin{equation*}
    A(\theta) = \frac{1}{2} \norm{\theta}^2 + \ConvexIndicator{K}(\theta) \,,
\end{equation*}
\Cref{eq:expofam-type} becomes
\begin{equation*}
    T(x) = \argmin_{\theta \in K} \frac{1}{2} \norm{x - \theta}^2 + h_C(\theta)
    \,.
\end{equation*}

\subsection{L1 penalized estimators of symmetric matrices}
Penalization by the $\ell_1$ norm is a well-known method for inducing sparsity 
in estimates.  It corresponds to taking the penalty support set to be an 
$\ell_\infty$ ball, i.e. 
\begin{equation*}
    C = \set{z \given \norm{z}_\infty \leq \lambda}
    \,,
\end{equation*}
with $\lambda \geq 0$.
Combining this with the least squares leads to a special 
case of the estimator known as the Lasso \citep{tibshirani:regression}. 
Here we give four further examples that involve estimating a symmetric matrix 
from a symmetric matrix input:
\begin{equation*}
    X \in \mathcal{X} 
    = \Sym_p 
    \coloneqq \set{ x \in \Real^{p \times p} \given x = x^T} 
    \,.
\end{equation*}
In all four cases, the set $C$ is taken to be
\begin{equation*}
    C = \set{z \in \Sym_p \given \norm{z}_\infty \leq \lambda}
    \,,
\end{equation*}
which makes $h_C = \norm{}_1$ the entrywise $\ell_1$ norm of a symmetric 
matrix.
Alternatively, we could consider a weighted version 
\begin{equation*}
    C 
    = \set{z \in \Sym_p \given \abs{z_{ij}} \leq \lambda_{ij}
        \text{ for all $i,j$}
    }\,,
\end{equation*}
with $\lambda_{ij} \geq 0$. For example, this could be used to avoid 
penalizing the diagonal by setting $\lambda_{ii} = 0$.

\begin{example}[Graphical Lasso]
\label{example:glasso}
The Graphical Lasso \Cref{eq:glasso} corresponds to selecting
\begin{equation*}
    A(\theta) = 
    \begin{cases}
        -\log\det(-\theta) &\text{if $-\theta \succeq 0$} \\
        +\infty & \text{otherwise.}
    \end{cases}
\end{equation*}
Note that we have reversed the sign of $\theta$ in this formulation so 
that
\begin{equation*}
    T(X) 
    = \argmin_\theta -\log\det(-\theta) - \innerp{X}{\theta} + \lambda \norm{\theta}_1
    \,,
\end{equation*}
but we could have instead replaced $X$ by $-X$ in \Cref{eq:glasso}.
\end{example}

\begin{example}[Sparse PCA via Fantope Projection]
\label{example:sparse_pca}
Sparse PCA via Fantope Projection \Cref{eq:fps} 
corresponds to choosing $A$ to be the convex indicator function of the 
Fantope,
\begin{equation*}
    A(\theta) 
    = \ConvexIndicator{\mathcal{F}^k}(\theta) 
    = 
    \begin{cases}
        0 
        & \text{if $0 \preceq \theta \preceq \id$ and $\trace(\theta) = k$,} 
        \\
        +\infty & \text{otherwise.}
    \end{cases}
\end{equation*}
As mentioned in \Cref{sec:example}, $\mathcal{F}^k$ is the convex hull 
of rank-$k$ projection matrices and the optimization problem can be viewed 
as a convex relaxation of $\ell_1$ penalized variance maximization.
\end{example}

\begin{example}[Sparse covariance estimation with eigenvalue constraints]
\label{eq:sparse_covariance}
Sparse covariance estimation by minimizing an $\ell_1$ penalized Gaussian 
log-likelihood does not lead to a convex optimization problem. As 
an alternative, \citet{xue.ma.ea:positive,liu.wang.ea:sparse} have 
proposed using least squares with a constraint on the smallest eigenvalue to 
ensure positive definiteness. Their estimator fits into our framework by 
taking 
\begin{equation*}
    A(\theta) 
    = \frac{1}{2} \norm{\theta}^2 
    + 
    \begin{cases}
        0 & \text{if $\theta \succeq \epsilon \id$, and} \\
        +\infty & \text{otherwise,}
    \end{cases}
\end{equation*}
were $\epsilon > 0$ is a lower bound on the smallest eigenvalue of the 
estimate to ensure positive definiteness.  The resulting procedure 
takes a sample covariance matrix $X$ as input and is equivalent to
\begin{equation*}
    T(X) = \argmin_{\theta \succeq \epsilon \id} 
    \frac{1}{2} \norm{X - \theta}^2 + \lambda \norm{\theta}_1
    \,.
\end{equation*}
This is a special case of \Cref{example:least_squares}, and 
can be viewed as an $\ell_1$ penalized projection of $X$ onto a closed 
subset of the positive semidefinite cone.
\end{example}

\begin{example}[$\ell_1$ penalized Ising model selection]
\label{example:ising}
The Ising model is an attractive exponential family model for multivariate 
binary data, but the $\ell_1$ penalized likelihood approach has largely been 
avoided due to the computational intractability of its log-partition 
function, 
\begin{equation*}
    A(\theta) = \log \sum_{u \in \set{-1,+1}^p} \exp(\innerp{u u^T}{\theta})
    \,.
\end{equation*}
Instead, there have been proposals of alternative methods such as  
pseudo-likelihood \citep{hofling.tibshirani:estimation}, composite 
conditional-likelihood \citep{xue.zou.ea:nonconcave}, and local 
conditional-likelihood \citep{ravikumar.wainwright.ea:high-dimensional}. 
Leaving aside the computational issue for now, we recognize that 
the penalized maximum likelihood estimator based on the above $A$ 
falls into our framework.
\end{example}

\subsection{Group Lasso and other norms}
The group Lasso \citep{yuan.lin:model*1} is a block-structured generalization of the Lasso.  It imposes sparsity on blocks of entries of 
$\theta$ rather than on individual entries. For example, suppose that 
the entries of $\theta$ are partitioned into $m$ blocks as 
$\theta = (\theta_{B_1}, \theta_{B_2}, \ldots, \theta_{B_m})$.  The group Lasso penalty is defined as 
\begin{equation*}
    \sum_{j=1}^m \lambda_j \norm{\theta_{B_j}} \,.
\end{equation*}
This corresponds to the penalty support set
\begin{equation}
\label{eq:group_lasso}
    C 
    = \set[\Big]{z \given \norm{z_{B_j}} \leq \lambda_j \text{ for all $j$}}
    \,.
\end{equation}
The group Lasso penalty is itself a norm and more generally, if $\nu$ 
is a norm, then 
\begin{equation*}
    \nu(\theta) = \max\set{\innerp{z}{\theta} \given \nu_*(z) \leq 1}
    \,,
\end{equation*}
where
\begin{equation*}
    \nu_*(z) = \max\set{\innerp{y}{z} \given \nu(y) \leq 1}
\end{equation*}
is the dual norm. So $\nu$ is the support function of the unit ball of its 
dual norm.

\subsection{Cone constraints}
Methods that employ order restrictions such as isotonic regression 
\citep{barlow.bartholomew.ea:statistical} or those that employ 
positivity constraints can be viewed as special cases of requiring 
that $\theta$ lie in a closed convex cone $K$. For example, 
\citet{slawski.hein:estimation,lauritzen.uhler.ea:maximum} studied the 
estimation of the inverse covariance matrix of a multivariate Gaussian 
under the assumption that its off-diagonal elements are all nonnegative. 
This corresponds to the cone of symmetric matrices with nonnegative 
off-diagonal entries: 
\begin{equation*}
    K_{\geq 0}
    = \set{u \in \Sym_p \given u_{ij} \geq 0 \text{ for all $i \neq j$}}
    \,.
\end{equation*}
To incorporate a closed convex cone constraint into 
\Cref{eq:expofam-type}, 
we could add the convex indicator of the cone to $A$. For the inverse 
covariance estimator with the Gaussian log-likelihood, we reverse the sign 
of $\theta$ as in \Cref{example:glasso}, take $K = -K_{\geq 0}$ and 
\begin{equation*}
    A(\theta) 
    = -\log\det(-\theta) + \ConvexIndicator{K}(\theta)
    \,,
\end{equation*}
This induces a nonnegativity constraint on the off-diagonals of $-\theta$. 
We could also incorporate this constraint into $C$. 
In general, the convex indicator of a closed convex cone $K$ is equal to the 
support function of its polar, 
\begin{equation*}
    K^\circ = \set{ u \given \innerp{u}{v} \leq 0 \text{ for all } v \in K}
    \,,
\end{equation*}
i.e. $\ConvexIndicator{K} = h_{K^{\circ}}$ 
\citep[Example C.2.3.1]{HUL:fundamentals}. 
Then using the fact that the sum of support functions is the support 
function of the sum of the sets \citep[Proposition C.2.2.1,Theorem C.3.3.2]{HUL:fundamentals}, 
\begin{equation*}
    \ConvexIndicator{K}(\theta) + h_C(\theta)
    = 
    h_{K^\circ}(\theta) + h_C(\theta)
    =
    h_{K^\circ + C}(\theta)
    \,.
\end{equation*}
So there is some flexibility in how constraints are represented in 
this framework. 
\section{Group invariance and convexity}
\label{sec:group_invariance}
The generators of expofam-type estimators often have symmetries. For 
example, Graphical Lasso (\Cref{example:glasso}), Sparse PCA via Fantope Projection (\Cref{example:sparse_pca}), and the 
sparse covariance estimator in \Cref{eq:sparse_covariance} all satisfy
\begin{equation*}
    A(\theta) = A(U \theta U^{-1})
\end{equation*}
whenever $U$ is an orthogonal matrix.  
Least squares satisfies 
\begin{equation*}
    A(\theta) = A(U \theta)
\end{equation*}
for all orthogonal matrices $U$.  The generator of the Ising model 
(\Cref{example:ising}) is invariant under conjugation by diagonal 
sign matrices:
\begin{equation*}
    A(\theta) = A(D \theta D^{-1})
\end{equation*}
for all diagonal matrices $D$ with entries $\pm 1$ along their diagonal. Since 
these matrices are orthogonal, this invariance holds for the previously 
mentioned examples as well.
These symmetries are important, because they tell us about the contours of $A$.
We can express such symmetries in terms of a group of transformations.
Let $\Group$ be a compact subgroup of the orthogonal group $\Orthogonal(\X)$ of $\X$ acting linearly on $\X$.%
\footnote{The restriction to $\Group \subseteq \Orthogonal(\X)$ ensures that %
the inner product is $\Group$-invariant, %
i.e. $\innerp{g \cdot x}{g \cdot y} = \innerp{x}{y}$.} %
A function $f$ on $\X$ is $\Group$-invariant if it is invariant under 
the action of $\Group$ on $\X$, i.e. $f(g \cdot x) = f(x)$ for all 
$x \in \X$ and $g \in \Group$.

\subsection{Lower level sets, orbitopes, and group majorization}
An expofam-type estimator $T$ with generator $A$ and penalty support set $C$ 
can be computed by minimizing the sum of three terms: 
$A(\cdot)$, $-\innerp{x}{\cdot}$, and $h_C(\cdot)$. Let us focus 
temporarily on the first term and suppose that $A$ is $\mathcal{G}$-invariant. 
Fix any $u \in \X$; think of it as a candidate for the optimization 
problem. Note that the lower level sets of $A$ are also 
$\mathcal{G}$-invariant:
\begin{equation*}
    \set{v \given A(v) \leq A(u)}
    = 
    \set{v \given A(g \cdot v) \leq A(u)}
    = 
    g^{-1} \cdot \set{v \given A(v) \leq A(u)}
\end{equation*}
for all $g \in \Group$. 
In particular, the orbit of $u$ under $\Group$ satisfies
\begin{equation*}
    \Group \cdot u \subseteq \set{v \given A(v) \leq A(u)} 
    \,.
\end{equation*}
Since $A$ is closed and convex, its lower level sets are also closed 
and convex, and hence 
\begin{equation}
    \label{eq:orbitope_inclusion}
    \conv( \Group \cdot u ) 
    \subseteq 
    \set{v \given A(v) \leq A(u)}
    \,.
\end{equation}
The left-hand side of \Cref{eq:orbitope_inclusion} is the convex hull 
of the orbit of $\theta$ under $\Group$, and is called the 
\emph{orbitope} of $\Group$ with respect to $\theta$ 
\citep{sanyal.sottile.ea:orbitopes}. It is compact, because $\Group$ is 
compact. The inclusion \Cref{eq:orbitope_inclusion} is remarkable, 
because the orbitope depends only on $\Group$ and $\theta$, so 
\Cref{eq:orbitope_inclusion} holds \emph{simultaneously} for all 
$\Group$-invariant $A$. 

We can improve the value of $A$ by moving from $u$ to any point in the 
orbitope, but to do this we need to be able to identify elements of the 
orbitope.  That is, given $u, v \in \X$, we need to be able to determine if 
\begin{equation*}
    v \in \conv(\Group \cdot u) \,.
\end{equation*}
This relation is known as $\Group$-majorization 
\citep{eaton.perlman:reflection} and it induces a preorder 
(reflexive and transitive) on $\X$:
\begin{equation*}
    v \preceq_{\Group} u 
    \iff
    v \in \conv(\Group \cdot u)
    \iff
    \conv(\Group \cdot v) \in \conv(\Group \cdot u)
    \,.
\end{equation*}
The rightmost equivalence follows from the $G$-invariance and convexity of 
the orbitope. When the above holds, we say that $u$ $\Group$-majorizes $v$. 
Consequently, \Cref{eq:orbitope_inclusion} implies that $A$ is 
$\Group$-monotone:
\begin{equation*}
    v \preceq_{\Group} u \implies A(v) \leq A(u) \,.
\end{equation*}
To \emph{find} a point in the orbitope, suppose that there is a map 
$Q : \X \to \X$ satisfying
\begin{equation}
    \label{eq:averaging}
    Q u \preceq_{\Group} u \quad \text{for all} \quad u \in \X \,.
\end{equation}
Then $A(Q u) \leq A(u)$. So if we have such a map, then it ``solves'' the 
problem of improving the value of $A$, but to apply $Q$ to 
\Cref{eq:expofam-type} we will need to consider the other terms.

\subsection{Reduction of the parameter space}
Expressing \Cref{eq:expofam-type} in saddle point form,
\begin{equation}
    \label{eq:saddle_point}
    \min_\theta A(\theta) - \innerp{x}{\theta} + h_C(\theta)
    = \min_\theta \max_{z \in C} A(\theta) - \innerp{x - z}{\theta} \,,
\end{equation}
we would like to replace $\theta$ by $Q \theta$ but in such a way 
that the objective does not increase. With this foresight, 
suppose that $Q$ is linear and that its adjoint satisfies 
\begin{equation}
    \label{eq:dual_feasibility}
    Q^* (x - C) \subseteq x - C \,.
\end{equation}
Then by \Cref{eq:averaging,eq:dual_feasibility},
\begin{align*}
    \max_{z \in C} A(\theta) - \innerp{x - z}{\theta} 
    &\geq
    \max_{z \in C} A(Q \theta) - \innerp{Q^*(x - z)}{\theta} \\
    &=
    \max_{z \in C} A(Q \theta) - \innerp{x - z}{Q \theta} \\
    &\geq
    \min_\theta \max_{z \in C} A(\theta) - \innerp{x - z}{\theta} \,,
\end{align*}
because $Q \X \subseteq \X$. Now if $\theta_* \in T(x)$,  
then we can substitute it for $\theta$ above to obtain an equality, 
\begin{equation*}
    \min_\theta \max_{z \in C} A(\theta) - \innerp{x - z}{\theta}
    =
    \max_{z \in C} A(Q \theta_*) - \innerp{x - z}{Q \theta_*} 
    \,.
\end{equation*}
This implies that $Q \theta_* \in T(x)$, and we have proven the following 
theorem.
\begin{theorem}
\label{thm:parameter_reduction}
Let $T$ be an expofam-type estimator with 
a generator $A$ that is closed, convex, proper and $\Group$-invariant and 
penalty support set $C$. Fix $x \in \X$. 
If $Q: \X \to \X$ is a linear map satisfying
\begin{enumerate}
    \item (averaging) $Q u \preceq_G u$ for all $u \in \X$, and 
    \item (dual feasibility) $Q^*(x - C) \subseteq x - C$, 
\end{enumerate}
then $Q T(x) \subseteq T(x)$. 
Moreover, if $T$ is at most singleton-valued, then $Q T(x) = T(x)$.
\end{theorem}

\subsection{Consequences}
The power of \Cref{thm:parameter_reduction} is that it applies generically 
to expofam-type estimators---it depends only on symmetries of their generator and the penalty support set. So rather than starting from a specific 
$T$, we could instead start from a compact subgroup $\Group$ of the orthogonal 
group and a closed convex set $C$.  Then \Cref{thm:parameter_reduction}
immediately leads to the following corollary.
\begin{corollary}
\label{cor:simultaneous_parameter_reduction}
Let $\Group \subset \Orthogonal(\X)$ be a compact subgroup, $C \subseteq \X$ 
be a nonempty closed convex set, and consider the collection $\mathcal{M}$ of 
all expofam-type estimators with a $\Group$-invariant generator and   
penalty support set $C$.  Fix $x \in \X$. If $Q : \X \to \X$ is a linear map satisfying the averaging and dual feasibility conditions of 
\Cref{thm:parameter_reduction}, then for all $T \in \mathcal{M}$ we have 
that $Q T(x) \subseteq T(x)$ and with equality if $T(x)$ is a singleton.
\end{corollary}
\Cref{cor:simultaneous_parameter_reduction} has several consequences. 
A practical consequence is that given an input $x$, 
if we can construct a $Q$ satisfying 
the conditions above, then we can reduce the optimization problem underlying 
every $T \in \mathcal{M}$. Each such $T$ has a solution in 
the range of $Q$, so we can 
construct $Q$ once and then optimize over its range rather than the entirety 
of $\X$ for each $T \in \mathcal{M}$. 
A theoretical consequence of \Cref{cor:simultaneous_parameter_reduction} is 
that it provides a new way to reason about the solutions of an optimization 
problem. For example, it is often of interest to 
determine conditions on $x$ that ensure $T(x)$ lies in some subspace, 
e.g. model selection consistency. This perspective relates \Cref{cor:simultaneous_parameter_reduction} to the primal-dual witness technique 
\citep{wainwright:sharp} which has been succesfully applied to the 
analysis of a large variety of sparse estimators. The advantage of 
\Cref{cor:simultaneous_parameter_reduction} is that it relies only on symmetry 
properties of the generator and so it holds simultaneously for all 
$T \in \mathcal{M}$. 
\section{Computationally sufficient reductions}
\label{sec:reduction}
The previous section shows that it may be possible to reduce the parameter 
space of procedures that are expofam-type estimators. In this section we will 
show how to build on \Cref{thm:parameter_reduction} to reduce the input space 
as well.  The main results are \Cref{thm:projection_reduction} and its 
corollary below; the theorem gives additional conditions for strengthening the 
result of the previous section to
\begin{equation*}
  Q T(x) = Q T(Q x) \subseteq T(x) \cap T(Q x)\,.
\end{equation*}
The result reveals a sort of duality between 
reducing the parameter space and reducing the input space for expofam-type 
estimators. \Cref{cor:csr} then shows how to exploit this to construct a 
computationally sufficient reduction. As with \Cref{thm:parameter_reduction}, 
this hinges on being able to construct suitable maps $Q$.  So the last part 
of the section is devoted to discussing a strategy and some examples.

\begin{theorem}[Reduction by projection]
\label{thm:projection_reduction}
Let $T$ be an expofam-type estimator with penalty support set $C$, 
and a generator $A$ that is closed, convex, proper and $\Group$-invariant. 
Fix $x \in \X$. If $Q: \X \to \X$ is an orthogonal projection satisfying
\begin{enumerate}
    \item (averaging) $Q u \preceq_G u$ for all $u \in \X$, 
    \item (dual feasibility) $Q (x - C) \subseteq x - C$, and
    \item (dual invariance) $Q (x - C) \subseteq Q x - C$, 
\end{enumerate}
then 
\begin{equation*}
  Q T(x) = Q T(Q x) \subseteq T(x) \cap T(Q x)
\end{equation*}
In particular, if $T(x)$ is a singleton, then 
\begin{equation*}
  Q T(x) = Q T(Q x) = T(x) = T(Q x) \,.
\end{equation*}
\end{theorem}
The proof is contained in \Cref{sec:proofs}, but to give some motivation, 
let us go completely through the saddle point formulation 
\Cref{eq:saddle_point} from the primal problem to the (Fenchel) dual problem, 
\begin{equation*}
  \min_\theta A(\theta) - \innerp{x}{\theta} + h_C(\theta) 
  = -\min_{w \in x - C} A^*(w) 
  \,.
\end{equation*}
$A^*$ is the convex conjugate of $A$. If $A$ is $\Group$-invariant, 
then so is $A^*$.  So we can try to exploit $\Group$-monotonicity.
Although the proof of the theorem does not explicitly use the 
dual, the gist of it is that we want to ensure that feasible set, $x - C$, of 
the dual problem can be replaced by 
$Q x - C$. That is the rationale behind the \emph{dual invariance} condition. 
The next lemma gives some simpler conditions to ensure that dual invariance 
holds. Its proof is in \Cref{sec:proofs}.
\begin{lemma}
\label{lem:dual_invariance_easy}
In \Cref{thm:projection_reduction}, if $Q C \subseteq C$, then 
dual invariance is satisfied. If $C$ is $\Group$-invariant, then dual 
invariance is implied by averaging.
\end{lemma} 

\subsection{A computationally sufficient reduction}
\Cref{thm:projection_reduction} guarantees that for each fixed $x$, 
if we can construct an orthogonal projection $Q_x$ satisfying the conditions 
of the theorem, then $Q_x T(x) = Q_x T(Q_x x) \subseteq T(x)$. If we let 
$S(x) = (Q_x, Q_x x)$, then $S$ is clearly computationally sufficient. 
However, this is not too useful, because $Q_x$ can depend on $x$ in a nontrivial way and it is not clear if $Q_x$ is a meaningful reduction of $x$. 
We would rather have that $R(x) = Q_x x$ alone be computationally sufficient. 
The main difficulty with applying \Cref{thm:projection_reduction} is that 
$Q_x T(R(x)) \subseteq T(x)$, but given $T(R(x))$ how do we find an element of 
$Q_x T(R(x))$ without relying on $Q_x$? The following proposition shows that 
this can be done by finding the minimum norm element.
\begin{proposition}
\label{pro:minimum_norm}
Let $B$ be a nonempty closed convex set and $P$ be an orthogonal projection 
that leaves $B$ invariant. Then $B$ has a unique minimum norm element 
$\theta_*$ and $P\theta_* = \theta_*$.
\end{proposition}
\begin{proof}
Since $B$ is a nonempty closed convex set, it has a unique minimum norm element 
$\theta_*$---this is the metric projection of $0$ onto $B$ 
\citep[see, e.g.,][Theorem 3.16]{bauschke.combettes:convex}. Now 
$P \theta_* \in B$ and $\norm{P \theta_*} \leq \norm{\theta_*}$, because 
$P$ is an orthogonal projection. Since $\theta_*$ is the unique minimum norm 
element of $B$, it follows that $P \theta_* = \theta_*$.
\end{proof}
\begin{corollary}
\label{cor:csr}
Let $\Group \subset \Orthogonal(\X)$ be a compact subgroup, $C \subseteq \X$ 
be a nonempty closed convex set, and consider the collection $\mathcal{M}$ of 
all expofam-type estimators with a $\Group$-invariant generator and   
penalty support set $C$.  For each $x \in \X$, suppose that $Q_x : \X \to \X$ is an orthogonal projection satisfying the conditions of 
\Cref{thm:projection_reduction}. Then the function $R(x) = Q_x x$ is computationally sufficient for $\mathcal{M}$.
\end{corollary}
\begin{proof}
Let $T \in \mathcal{M}$. Note 
that $T(y)$ is closed and convex for each $y \in \X$, 
because $T(y)$ is the set of minimizers of a closed convex function. 
So the set-valued function 
\begin{equation*}
  f(T, y) = \argmin_{\theta \in T(y)} \norm{\theta} 
\end{equation*}
is at most singleton-valued. For each $x \in \X$, 
$Q_x$ satisifies the conditions of \Cref{thm:projection_reduction}. Thus, 
\begin{equation*}
  Q_x T(R(x)) = Q_x T(x) \subseteq T(x) \cap T(R(x)) \,.
\end{equation*}
$T(R(x))$ is nonempty if and only if $T(x)$ is nonempty, so
\begin{equation*}
  f(T, R(x)) \neq \emptyset \iff T(x) \neq \emptyset \,.
\end{equation*}
Moreover, $T(R(x))$ is closed convex and invariant under $Q_x$, so 
\Cref{pro:minimum_norm} implies that 
\begin{align*}
  f(T, R(x))
  = Q_x f(T, R(x)) 
  \subseteq Q_x T(R(x)) 
  \subseteq T(x) \,.
\end{align*}
Thus, $R$ is computationally sufficient for $\mathcal{M}$.
\end{proof}

\subsection{Constructing a reduction}
\Cref{cor:csr} 
gives sufficient conditions for constructing a computationally sufficient 
reduction.  We first need to identify a group $\Group$ and penalty support 
set $C$. Then there are three conditions: averaging, dual feasibility, and 
dual invariance.  \Cref{lem:dual_invariance_easy} gives cases where 
dual invariance is automatically satisfied. So given a collection of 
expofam-type estimators $\mathcal{M}$, we take the following steps.
\begin{enumerate}
\item Identify the orbitopes of $\Group$. This will help us determine when averaging holds and may suggest the form of the projections $Q_x$.
\item For each $x$, determine projections $Q_x$ such that 
      $Q_x u \in \conv(\Group \cdot u)$ for all $u$. This is averaging.
\item Verify that $Q_x (x - C) \subseteq x - C$. This is dual feasibility.
\item Dual invariance is automatically satisfied if $C$ is $\Group$-invariant 
or if $Q_x C \subseteq C$. Otherwise, verify that 
$Q_x(x - C) \subseteq Q_x x - C$.
\end{enumerate}
Each of these steps can be very involved and may require luck. Even the first 
step of identifying the orbitopes can be challenging. Many orbitopes are known, but due to limitations of space and scope we will not list any beyond those used in the examples. The existing literature on $\Group$-majorization 
\citep[see][]{eaton.perlman:reflection} and \citet{sanyal.sottile.ea:orbitopes} are good starting points for further exploration. 
See also \citet{negrinho.martins:orbit}. 

\subsection{Examples}
In this section we will work through three simple examples to demonstrate 
the strategy enumerated above. In all three cases we will keep the group 
$\Group$ fixed to be the group of sign symmetries. The application of 
the machinery developed in the preceding sections may seem like overkill 
for these examples, but the main point is to understand how different 
penalties interact with the group, because we will see similar patterns return in a more sophisticated form when we work on our main example in 
\Cref{sec:single_linkage_symmetry}.

\begin{example}[L1 penalties]
\label{sec:sign_symmetry_lasso}
Let $\X = \Real^n$ and $\mathcal{M}$ be the collection of expofam-type 
estimators with generators $A$ satisfying
\begin{equation*}
  A(D \theta) = A(\theta)
\end{equation*}
for all diagonal sign matrices $D$ and with penalty support set
\begin{equation*}
  C = 
  \set{z \in \Real^n 
  \given \abs{z_i} \leq \lambda_i, \text{ for } i=1,\ldots, n }
\end{equation*}
with $\lambda_i \geq 0$ for all $i$. 
In this case, $h_C$ is a weighted $\ell_1$ norm, and $\mathcal{M}$ includes 
the Lasso:
\begin{equation}
  \label{eq:lasso}
  \argmin_\theta \frac{1}{2} \norm{x - \theta}^2 
  + \sum_i \lambda_i \abs{\theta_i} \,.
\end{equation}
The group $\Group$ acts on $u \in \X$ by multiplying each entry by $\pm 1$, e.g. 
\begin{equation*}
  g \cdot u = g \circ u \,,
\end{equation*}
where $g \in \set{-1, +1}^n$ and $\circ$ denotes entrywise multiplication.
The orbitope is easily seen to be 
\begin{equation*}
  \conv(\Group \cdot u) 
  = \conv \set{d \circ u \given \norm{d}_\infty \leq 1 } \,.
\end{equation*}
So 
\begin{equation*}
  v \preceq_{\Group} u 
  \iff
  v = d \circ u \quad \text{with} \quad \norm{d}_\infty \leq 1 \,.
\end{equation*}
The map $u \mapsto d \circ u$ 
is linear and self-adjoint. It is idempotent if and only if  
$d \in \set{0, 1}^n$. So we will consider maps $Q$ of the form 
\begin{equation*}
  Q u = d \circ u
\end{equation*}
with $d \in \set{0, 1}^n$. 
For each coordinate $i$, the dual feasibility condition reduces to 
\begin{equation*}
  d_i = 0 \implies \abs{x_i} \leq \lambda_i \,.
\end{equation*}
There is some flexibility here.  At one extreme we could take $d_i = 1$ 
for all coordinates, but that would not provide any reduction. Instead, we 
make the dual feasibility condition tight by setting 
$d_i = 0 \iff \abs{x_i} \leq \lambda_i$. The resulting map is the 
hard-thresholding operator, 
\begin{equation*}
  [R(x)]_i =
  \begin{cases}
    x_i & \text{if $\abs{x_i} > \lambda_i$} \\
    0 & \text{otherwise.} 
  \end{cases}
\end{equation*}
The last condition to check is dual invariance. 
Since $C$ is $\Group$-invariant,
dual invariance is automatically satisfied (\Cref{lem:dual_invariance_easy}. 
So we have successfully shown that $R$ 
is computationally sufficient for $\mathcal{M}$.  

We can also establish computational minimality of $R$. 
Let $U$ be computationally 
sufficient for $\mathcal{M}$ and let $T$ be the Lasso \Cref{eq:lasso}. 
Since $T \in \mathcal{M}$, $T(x)$ is essentially a function of $U(x)$. 
So it is enough for us to show that $R(x)$ can be computed from $T(x)$. 
In this simple setting, the Lasso is actually the same as the 
soft-thresholding operator:
\begin{equation*}
  [T(x)]_i = 
  \begin{cases}
    x_i - \lambda_i \sign(x_i) & \text{if $\abs{x_i} > \lambda_i$} \\
    0 & \text{otherwise.}
\end{cases}
\end{equation*}
Then clearly, 
\begin{equation*}
  [R(x)]_i = 
  \begin{cases}
    [T(x)]_i + \lambda_i \sign([T(x)]_i) & \text{if $[T(x)]_i \neq 0$} \\
    0 & \text{otherwise.}
  \end{cases}
\end{equation*}
So $R$ is computationally minimal for $\mathcal{M}$. Notice however that 
this argument also shows that $T$ is computationally sufficient. Since 
$T \in \mathcal{M}$, it follows from \Cref{lem:trivially_minimal} that 
$T$ must also be computationally minimal.  So every procedure in 
$\mathcal{M}$ can simply be viewed as a refinement of hard-thresholding 
or, equivalently, soft-thresholding.
\end{example}

\begin{example}[Group Lasso]
\label{sec:sign_symmetry_group_lasso}
This next example extends the previous by considering expofam-type 
estimators on $\X = \Real^n$ with the Group Lasso penalty.  
Let $B_1,\ldots,B_m$ be a 
partition of $[n] = \set{1,\ldots,n}$. We continue to assume 
that the generators of $\mathcal{M}$ satisfy 
\begin{equation*}
  A(D \theta) = A(\theta)
\end{equation*}
for diagonal sign matrices, but now we take the penalty support set to be 
\begin{equation*}
  C = 
  \set{z \in \Real^n 
  \given \norm{z_{B_i}} \leq \lambda_i, \text{ for } i=1,\ldots,m }
\end{equation*}
with $\lambda_i \geq 0$. This corresponds to the Group Lasso penalty.
We have already discussed the group $\Group$ and orbtiope in 
\Cref{sec:sign_symmetry_lasso}. We will again consider maps of the form 
\begin{equation*}
  Q u = d \circ u \,,
\end{equation*}
with $d \in \set{0, 1}^n$. For a block of indices $B_i$, the dual 
feasibility condition holds if 
\begin{equation*}
  d_{B_i} = 0 \implies \norm{ x_{B_i} } \leq \lambda_i \,.
\end{equation*}
To make this tight, we set
\begin{equation*}
  d_{B_i} = 0 \iff \norm{ x_{B_i} } \leq \lambda_i \,.
\end{equation*}
Since the penalty support set $C$ for the Group Lasso is also 
$\Group$-invariant, dual invariance holds automatically. Thus, the 
blockwise hard-thresholding operator
\begin{equation*}
  [R(x)]_{B_i} = 
  \begin{cases}
    x_{B_i} & \text{if $\norm{ x_{B_i} } > \lambda_{B_i}$} \\
    0 & \text{otherwise} 
  \end{cases}
\end{equation*}
is computationally sufficient.  This is essentially the same as the 
previous example.  Using exactly the same technique as before, it can be shown 
that $R$ is computationally minimal.
\end{example}

\begin{example}[Positivity constraints]
In this final example consider expofam-type estimators on $\X = \Real^n$ 
with positivity constraints.  We will incorporate this by taking the 
penalty support set to be the polar of the nonnegative cone, i.e. 
\begin{equation*}
  C = \set{z \in \Real^n \given z_i \leq 0 \text{ for } i=1,\ldots,n }
\end{equation*}
so that
\begin{equation}
  h_C(\theta) = 
  \begin{cases}
    0 &\text{if $\theta_i \geq 0$ for all $i$} \\
    +\infty &\text{otherwise.}
  \end{cases}
\end{equation}
We will once again assume that the generators $A$ are sign symmetric, i.e. 
$A(D\theta) = A(\theta)$ for all diagonal sign matrices. Note that in this 
example, the penalty support set $C$ \emph{is not} $\Group$-invariant. 
That is the main point of this example. We have already determined the 
orbitope and the form of the projection $Q u = d \circ u$ in the previous two examples. 
The dual feasibility condition reduces to 
\begin{equation*}
  d_i = 0 \implies x_i \leq 0 \,.
\end{equation*}
To make it tight we will choose $d_i = 0 \iff x_i \leq 0$.
To verify dual invariance, note that $d \in \set{0,1}^n$ and  
\begin{equation*}
  d \circ C \subseteq C \,.
\end{equation*}
Then dual invariance holds, and the 
computationally sufficient reduction that we have found is the 
positive part operator:
\begin{equation*}
  [R(x)]_i = \max(x_i, 0) \,.
\end{equation*}
We can easily demonstrate the minimality of $R$ by considering the 
nonnegative least squares estimator, 
\begin{equation*}
  T(x) = \argmin_{\theta \geq 0} \frac{1}{2} \norm{x - \theta}^2 \,.
\end{equation*}
This is an expofam-type estimator with a sign symmetric generator. 
Since $T$ is singleton-valued, 
\Cref{thm:projection_reduction} tells us that $T(x) = T(R(x))$, i.e. 
\begin{equation*}
  \argmin_{\theta \geq 0} \frac{1}{2} \norm{x - \theta}^2
  = 
  \argmin_{\theta \geq 0} \frac{1}{2} \norm{R(x) - \theta}^2
  \,.
\end{equation*}
Since $R(x) \geq 0$, it follows that $T(x) = R(x)$. Then 
\Cref{lem:trivially_minimal} implies that $T$ is computationally minimal.
\end{example}
 
\section{Single-linkage and switch symmetry}
\label{sec:single_linkage_symmetry}
Equipped with the tools from \Cref{sec:framework,sec:group_invariance,sec:reduction}, we are finally ready to return to our main example: 
the hidden connection between single-linkage clustering and the 
sparse multivariate methods shown in \Cref{sec:example}. 
Let $\X = \Sym_p$. The first step is to identify a group. In 
analogy with the examples from the previous section, consider the group 
$\Group$ of diagonal sign matrices acting on $\X$ by 
conjugation.  Then let $\mathcal{M}$ be the collection of expofam-type estimators on $\X$ with generators $A$ satisfying 
\begin{equation*}
  A(D \theta D^{-1}) = A(\theta)
\end{equation*}
for all diagonal sign matrices and with penalty support set
\begin{equation*}
  C 
  = \set{ Z \in \Sym_p 
    \given \abs{Z_{ij}} \leq \lambda \text{ for all } i,j } \,,
\end{equation*}
such that $\lambda \geq 0$. 
This includes all of the $\ell_1$ penalized 
symmetric matrix estimators presented in the 
earlier sections: Graphical Lasso, Sparse PCA via Fantope Projection, 
the sparse covariance estimator with eigenvalue constraints, and 
$\ell_1$ penalized Ising model selection. We will show that 
using the computational sufficiency reduction techniques developed earlier, 
we inevitably arrive at single-linkage clustering.

\subsection{Cut orbitope}
The first step is to identify the orbitopes  and the $\Group$-majorization. 
This is related to the following set,
\begin{equation*}
  \Cut_p = \conv(\set{y y^T \given y \in \set{-1, +1}^p })
\end{equation*}
which is called the \emph{cut polytope} \citep{laurent.poljak:on}.
The following lemma describes the orbitope. Its proof is in \Cref{sec:proofs}.
\begin{lemma}
\label{lem:cut_orbitope}
Let $\Group$ be the group of diagonal sign matrices acting on
$U,V \in \Sym_p$ by conjugation, i.e.
\begin{equation*}
  g \cdot U = D U D^{-1}  
\end{equation*}
with $g \in \Group$ represented by a diagonal matrix $D$ whose diagonal
entries are $\pm 1$. Then 
\begin{equation*}
  \conv(\Group \cdot U) = \set{B \circ U \given B \in \Cut_p}\,,
\end{equation*}
and hence $V \prec_{\Group} U$ if and only if $V = B \circ U$ for some 
$B \in \Cut_p$.
\end{lemma}
For any $B \in \Cut_p$, the map $U \mapsto B \circ U$ is 
linear and self-adjoint and, by \Cref{lem:cut_orbitope}, 
\begin{equation*}
  U \preceq_{\Group} B \circ U \,.
\end{equation*}
So it satisfies the averaging condition of \Cref{thm:projection_reduction}.
To ensure it is an orthogonal projection we will also require 
idempotence: $B \circ (B \circ U) = B \circ U$ for all $U$.  This holds  
if and only if $B$ is a binary matrix. The following proposition helps us 
identify such $B$. Its proof is also in \Cref{sec:proofs}.
\begin{proposition}
\label{pro:arcsin_cut}
Let
\begin{equation*}
  K = \set[\big]{(2/\pi) \arcsin[ \Sigma ] 
  \given \Sigma \succeq 0, \diag(\Sigma) = \one } \,,
\end{equation*}
where $\arcsin[\cdot]$ means that the function is applied entrywise.
Then $\conv(K) = \Cut_p$.
\end{proposition}
Note that if $\Sigma$ is a binary correlation matrix, then so is $\arcsin[\Sigma]$. Then it follows from 
\Cref{pro:arcsin_cut} that $\Cut_p$ contains all $p \times p$ 
binary correlation matrices. This leads us to consider projections of the form 
$U \mapsto B \circ U$ for $B$ that is a binary correlation matrix.

\subsection{Dual feasibility, ultrametrics, and single-linkage}
Dual invariance is automatically satisfied by \Cref{lem:dual_invariance_easy}, 
since $C$ is invariant under conjugation by diagonal sign matrices 
(\Cref{lem:dual_invariance_easy}). 
So all that remains is for us to verify dual feasibility. 
For a fixed input $X \in \Sym_p$, 
the dual feasibility condition is
\begin{equation}
  \label{eq:dual_feasibility_single_linkage}
  \abs{X_{ij}} > \lambda \implies B_{ij} = 1 \,.
\end{equation}
Setting $B_{ij}$ to $0$ everywhere else is not possible, 
because that could result in $B$ that is not a binary correlation matrix. 
To maximize the reduction we should minimize the number of nonzero entries of $B$ subject to the dual feasibility condition 
\Cref{eq:dual_feasibility_single_linkage} and the constraint that $B$ is a 
binary correlation matrix. This turns out to be 
related to \emph{ultrametric} matrices 
\citep{dellacherie.martinez.ea:inverse}. 
These are symmetric matrices $U$ that satisfy the ultrametric inequality
\begin{equation*}
  U_{ij} \geq \min(U_{ik}, U_{jk}) \quad \text{for all} \quad i,j,k \,.
\end{equation*}
The connection with symmetric binary correlation matrices is established 
in the following lemma, which is proved in \Cref{sec:proofs}.
\begin{lemma}
\label{lem:ultrametric_psd}
A symmetric binary matrix $B$ with ones along the diagonal is positive 
semidefinite if and only if it satisfies the ultrametric inequality.
\end{lemma}
In other words, a symmetric binary matrix with ones along the diagonal is a 
correlation matrix if and only if it is ultrametric. 
Therefore, to maximize the reduction we should minimize the number of nonzeroes among all $B$ that are ultrametric binary matrices with 
ones along the diagonal and that satisfy the dual feasibility criterion 
\Cref{eq:dual_feasibility_single_linkage}:
\begin{equation}
    \label{eq:slc_ultrametric}
    \begin{aligned}
        &\text{minimize}
          &   &
          \sum_{ij} B_{ij}
        \\
        &\text{subject to}
          &   & \text{$B$ is a binary ultrametric matrix, $B_{ii} = 1$, and} \\
        & &   & \abs{X_{ij}} > \lambda \implies B_{ij} = 1 \text{ for all } i,j
          \,.
    \end{aligned}
\end{equation}
This is related to the problem of finding a maximal subdominant ultrametric 
\emph{distance}, which is well-studied in the fields of numerical taxonomy 
\citep{jardine.jardine.ea:structure} and phylogenetics \citep[Theorem 7.2.9]{semple.steel.ea:phylogenetics}. The solution is given by single-linkage 
clustering which can be interpreted as producing both an ultrametric distance 
\citep{johnson:hierarchical} and a binary ultrametric matrix---the clustering matrix. The latter point of view will be established below. First, let 
us define single-linkage in a more convenient way. For a symmetric matrix 
$W$ and $\tau \in \Real$, let
\begin{equation*}
  [\SLC_\tau(W)]_{ij}
  \coloneqq
  \begin{cases}
    1 & \text{if $i = j$ or $\max_P \min_{uv \in P} W_{uv} > \tau$, and} \\
    0 & \text{otherwise,}
  \end{cases}
\end{equation*}
where the maximum is taken over all paths between $i$ and $j$ in the complete 
undirected graph on $[n]$. This is equivalent to the procedure described 
in \Cref{sec:example}.  To see this, the maxi-min criterion puts $i$ and $j$ 
in the same cluster if and only if there exists a sequence of links between 
$i$ and $j$ with weights $\abs{X_{ij}}$ all larger than $\tau$.  So the pair 
are connected by single links.
\begin{proposition}
\label{eq:slc_minimal_ultrametric}
$\SLC_\lambda(\abs{X})$ is the unique solution of \Cref{eq:slc_ultrametric}.
\end{proposition}
\begin{proof}
Let $Y = \SLC_{\lambda}(\abs{X})$. Clearly, $Y$ is dual feasible, symmetric 
and binary. 
To establish that $Y$ is a binary ultrametric, we only need to check the 
ultrametric inequality. Say that a path is admissible if the weights 
$\abs{X_{ij}}$ of the edges along the path are all strictly larger than 
$\lambda$.
Suppose that the ultrametric inequality is violated for a 
triplet $i,j,k$. Then $Y_{ij} = 0$ and $Y_{ik} = Y_{jk} = 1$. So   
there are admissible paths from $i$ to $k$ and from $j$ to $k$
and hence there is an admissible path from $i$ to $j$.  This contradicts the assumption that $Y_{ij} = 0$. So $\SLC_{\lambda}(\abs{X})$ must be an 
ultrametric matrix.

Next, let $U$ be any other binary ultrametric matrix
satisfying the constraints of \Cref{eq:slc_ultrametric} and suppose that 
there is $i,j$ such that $U_{ij} = 0$, 
but $Y_{ij} = 1$. If this is the case, then there 
must be an admissible path between $i$ and $j$, 
say $i = i_1,i_2,\ldots,i_m = j$. 
Those corresponding entries of $U$ must be $1$ (by the constraints of \Cref{eq:slc_ultrametric}) and if $U_{ij} = 0$, then 
by repeatedly applying the ultrametric inequality, 
\begin{align*}
  0 &= U_{ij} \\
  &\geq \min(U_{i_1i_2}, U_{i_2i_m}) \\
  &\geq \min(U_{i_1i_2}, U_{i_2, i_3}, U_{i_3i_m}) \\
  &\vdots \\
  &\geq \min( U_{i_1i_2}, \ldots, U_{i_{m-1} i_m}) = 1
  \,,
\end{align*}
which is a contradiction. So $U_{ij} = 1$ whenever $Y_{ij} = 1$, and hence 
\begin{equation*}
  \sum_{ij} U_{ij} \geq \sum_{ij} Y_{ij} 
  \,.
\end{equation*}
If equality is attained then we must have that $U = Y$.
\end{proof}
Now let
\begin{equation*}
  \SLT_\tau(W) = \SLC_\tau(\abs{W}) \circ W \,.
\end{equation*}
This is the single-linkage thresholding operator and we have shown in the 
above discussion that it is computationally sufficient. Thus, the phenomenon 
illustrated in \Cref{sec:example} is explained completely by the following 
theorem.
\begin{theorem}
Let $\mathcal{M}$ be a collection of expofam-type estimators on 
$X \in \Sym_p$ with 
generators $A$ that are invariant under conjugation by a diagonal matrix 
and suppose that their penalties are $h_C = \lambda \norm{}_1$. Then 
$\SLT_\lambda(X)$ is computationally sufficient for $\mathcal{M}$, 
and moreover every $T \in \mathcal{M}$ satisfies 
\begin{equation*}
  \SLC_\lambda(\abs{X}) \circ T(X) \subseteq T(X) \,.
\end{equation*}
\end{theorem}
The ``moreover'' part of the theorem is \Cref{thm:parameter_reduction}. 
The rest follows from our preceding discussion and \Cref{cor:csr}. 
We have thus far not been able to determine whether or not $\SLT_\lambda(X)$ 
is computationally minimality.  The only result towards the direction of 
minimality is \Cref{eq:slc_minimal_ultrametric}.

\subsection{Single-linkage and positivity constraints}
There is one more connection between symmetric matrix estimation and 
single-linkage clustering that we can point out.  
\citet{lauritzen.uhler.ea:maximum} studied maximum likelihood 
estimation of the inverse covariance matrix of a multivariate Gaussian 
distribution under a positivity restriction on its off-diagonal entries. 
They pointed out numerous connections with single-linkage clustering. Their use of ultrametrics inspired this author to do the same, but the most relevant 
connection to this article is their Proposition 3.6, which essentially 
establishes an exact thresholding phenomenon for their MLE.
Here we try to explain this connection in a more general setting.

We continue the setup from the first part of the section, but replace 
the penalty support set by the cone
\begin{equation*}
  C = \set{Z \in \Sym_p 
    \given Z_{ii} = 0, Z_{ij} \leq 0 \text{ for all } i,j}
  \,.
\end{equation*}
Let $\mathcal{M}$ be a collection of expofam-type estimators on 
$\Sym_p$ with generators invariant to conjugation by diagonal sign 
matrices and penalty set $C$ as above. This induces a positivity constraint on 
the off-diagonal entries of $T \in \mathcal{M}$.
The group and orbitope remain the same as before: diagonal sign matrices and 
cut orbitope. The only difference is that we will need to 
construct some different projections, then re-establish dual feasibility and 
dual invariance. Since the orbitope remains the same, we continue examining 
projections of the form $U \mapsto B \circ U$ for $B$ a binary correlation 
matrix.  Note that the cone $C$ is invariant under this map, so 
dual invariance holds (\Cref{lem:dual_invariance_easy}). This leaves us to 
verify dual feasibility, which for the positivity constraint becomes
\begin{equation*}
  X_{ij} > 0 \implies B_{ij} = 1 \,.
\end{equation*}
Arguing as before, $B$ must be a binary ultrametric metrix with ones along 
its diagonal, so $B = \SLC_0(X)$ is the best possible choice. Thus, 
\begin{equation*}
  \SLT_+(X) \coloneqq \SLC_0(X) \circ X
\end{equation*}
is computationally sufficient for $\mathcal{M}$. 
Moreover, we can also conclude that 
\begin{equation*}
  \SLC_0(X) \circ T(X) \subseteq T(X) \,.
\end{equation*}
So for any $T \in \mathcal{M}$, the set $T(X)$ has elements that are supported 
on $\SLC_0(X)$. 
The positive constrained Gaussian MLE studied by 
\Citeauthor{lauritzen.uhler.ea:maximum} is unique and so there is actually  
equality for that particular $T$ above. 
\section{Discussion}
\label{sec:discussion}
There is much that has been left out and not covered by this article. Here we 
point out some of those things, open problems,  and previews of work that may closely follow.

We have only discussed a relatively small number of examples 
in terms of groups and penalty support sets. However, any single collection 
of expofam-type estimators with generators obeying such group invariances 
must be fairly large and seemingly diverse---the main example in 
\Cref{sec:single_linkage_symmetry} includes PCA and the Ising model in the 
same 
collection. There may also be some criticism about the focus on sparsity. We 
would argue that sparsity or at least \emph{nondifferentiability} of $h_C$ is 
an important contributor to the existence of nontrivial reductions.

There are, however, immediate and important extensions of the examples 
given. For example, the extension of \Cref{sec:single_linkage_symmetry} to the case of 
asymmetric matrix estimators is fairly straightforward, but involved. It has 
implications for methods such as sparse singular value decomposition and 
biclustering. This will be addressed in a follow-up paper.

Linear modeling procedures such as ordinary least squares regression and 
generalized linear models also fit into the expofam-type framework, but it so 
far seems unlikely that considerations of group invariance will be useful in 
obtaining computationally sufficient reductions. There is already a large body 
of literature on so-called ``safe screening rules'' for regression procedures 
such as the Lasso \citep[see, e.g.,][]{el-ghaoui.viallon.ea:safe,tibshirani.bien.ea:strong,liu.zhao.ea:safe,wang.zhou.ea:lasso}. This area is highly relevant, but 
it has focused exclusively on reducing the parameter space rather than the 
input space.

The main example showed a deep connection between single-linkage clustering 
and estimators of symmetric matrices. There is clearly a 
monotonicity phenomenon in the the tuning parameter $\lambda$. This can be 
seen by direct examination, but it is not part of the general 
machinery. This article establishes a fair amount of machinery, and much of it remains to be exploited. The concept of computational minimality is 
appealing, but so far has been elusive to prove. A deep and interesting 
question that is left by \Cref{sec:single_linkage_symmetry} is whether or not 
single-linkage thresholding is computationally necessary. 
\paragraph{Acknowledgments}
This work was supported by the National Science Foundation under Grant No. DMS-1513621. Parts of this research were completed and inspired by visits of the author to different institutions. The author would like to thank Kei Kobayashi and the statistics group in the Mathematics Department at Keio University for their hospitality, conversations, and pointers. The author would also like to thank the Isaac Newton Institute for Mathematical Sciences for its hospitality during the Statical Scalability programme. Thanks also to Jing Lei and Yoonkyung Lee for comments and encouragement.
 
\appendix
\section{More on single-linkage}
\label{sec:single_linkage_extras}
Let $K$ be the complete undirected graph on vertices 
$\set{1,\ldots,p}$ with weights given by $\abs{x_{ij}}$. 
\citet{gower.ross:minimum} showed that the single-linkage dendrogram 
can be recovered from any maximal spanning tree (MST) of $K$---a 
subgraph of maximum weight connecting all of the vertices. 
Indeed, the steps of the single-linkage clustering 
algorithm, as described in \Cref{sec:single_linkage_cluster_analysis}, are 
equivalent to Kruskal's algorithm for finding an MST \citep{kruskal:on}. The algorithm proceeds by maintaining a forest to which it iteratively adds 
edges of maximum weight such that a cycle is not formed. The connected 
components of the intermediate forests correspond to cutting the 
dendrogram at successively smaller values of $\lambda$. Thus, 
\begin{equation}
    \label{eq:single_linkage_equivalence}
    \begin{aligned}
        &\text{$i$ and $j$ are in the same single-linkage cluster at level $\leq\lambda$} \\
        &\iff \text{there is a path from $i$ to $j$ consisting of edges with weights $> \lambda$,}
    \end{aligned}
\end{equation}
and the connected components induced by thresholding $X$ at level $\lambda$ 
correspond exactly to cutting the single-linkage dendrogram at height 
$\lambda$.  Given an MST of $K$, we can reconstruct the 
single-linkage dendrogram from top to bottom by successively removing the 
smallest weight edges from the MST.
 
\section{Additional proofs}
\label{sec:proofs}
\subsection{Proof of Theorem \ref{thm:projection_reduction}}
\begin{proof}
$Q$ is an orthogonal projection so it is self-adjoint and idempotent. 
We will use this fact repeatedly in the proof. 
The dual invariance condition and $\Group$-invariance of $A$ imply that 
\begin{equation}
  \label{eq:middle_link_raw}
  \begin{aligned}  
    A(\theta) - \innerp{Q x}{\theta} + h_C(\theta)
    &= \max_{z \in C} A(\theta) - \innerp{Q x - z}{\theta} \\
    &\geq \max_{z \in C} A(\theta) - \innerp{Q (x - z)}{\theta} \\
    &\geq \max_{z \in C} A(Q \theta) - \innerp{x - z}{Q \theta} \\
    &= A(Q \theta) - \innerp{x}{Q \theta} + h_C(Q \theta)
  \end{aligned}
\end{equation}
for all $\theta$ and hence 
\begin{align}
  \min_\theta
  A(\theta) - \innerp{Q x}{\theta} + h_C(\theta)
  \label{eq:middle_link_1}
  &\geq
  \min_\theta
  A(Q \theta) - \innerp{x}{Q \theta} + h_C(Q \theta) \\
  \label{eq:middle_link_2}
  &\geq
  \min_\theta
  A(\theta) - \innerp{x}{\theta} + h_C(\theta) \,.  
\end{align}
We will use these chains of inequalities for each direction of the proof.
Let $\theta_* \in T(x)$. \Cref{thm:parameter_reduction} guarantees that 
$Q \theta_* \in T(x)$ and so 
\begin{align*}
  \min_\theta A(\theta) - \innerp{x}{\theta} + h_C(\theta)
  &= A(Q \theta_*) - \innerp{x}{Q \theta_*} + h_C(Q \theta_*) \\ 
  &= A(Q \theta_*) - \innerp{Q x}{Q \theta_*} + h_C(Q \theta_*) \\
  &\geq \min_\theta A(\theta) - \innerp{Q x}{\theta} + h_C(\theta) \,.
\end{align*}
Appending \Cref{eq:middle_link_1,eq:middle_link_2} to this chain yields the 
equality, 
\begin{equation*}
  \min_\theta A(\theta) - \innerp{Q x}{\theta} + h_C(\theta)
   = A(Q \theta_*) - \innerp{Q x}{Q \theta_*} + h_C(Q \theta_*) \,,
\end{equation*}
and hence $Q \theta_* \in T(Q x)$. This proves that 
\begin{equation}
  \label{eq:reduction_inclusion_1}
  Q T(x) \subseteq T(Q x) \,. 
\end{equation}
Now let $\theta_* \in T(Q x)$. \Cref{thm:parameter_reduction} implies that 
\begin{align*}
  \min_\theta A(\theta) - \innerp{x}{\theta} + h_C(\theta) 
  &= 
  \min_\theta A(Q \theta) - \innerp{x}{Q \theta} + h_C(Q \theta) \\
  &= 
  \min_\theta A(Q \theta) - \innerp{Q x}{Q \theta} + h_C(Q \theta) \\
  &\geq 
  A(\theta_*) - \innerp{Q x}{\theta_*} + h_C(\theta_*) \,.
\end{align*}
Now we apply \Cref{eq:middle_link_raw} and then \Cref{eq:middle_link_2} to 
conclude that
\begin{equation*}
  \min_\theta A(\theta) - \innerp{x}{\theta} + h_C(\theta) 
  = A(Q \theta_*) - \innerp{x}{Q \theta_*} + h_C(Q \theta_*) 
\end{equation*}
and hence $Q \theta_* \in T(x)$. This proves that $Q T(Q x) \subseteq T(x)$. 
Now apply $Q$ to both sides of \Cref{eq:reduction_inclusion_1} to conclude that
\begin{equation}
  \label{eq:reduction_inclusion_2}
  Q T(x) \subset Q T(Q x) \subseteq T(x) \,.
\end{equation}
Applying $Q$ to both sides above once more yields
\begin{equation*}
  Q T(x) = Q T(Q x) \,.
\end{equation*}
Combining this equality with \Cref{eq:reduction_inclusion_1,eq:reduction_inclusion_2}, we have that 
\begin{equation*}
  Q T(x) = Q T(Q x) \subseteq T(x) \cap T(Q x) \,.
\end{equation*}
If $T(x)$ is a singleton, then clearly $Q T(x) = T(x)$ and $Q T(Q x) = T(Q x)$ 
so
\begin{equation*}
  Q T(x) = Q T(Q x) = T(x) = T(Q x) 
  \,.
  \qedhere
\end{equation*}
\end{proof}

\subsection{Proof of Lemma \ref{lem:dual_invariance_easy}}
\begin{proof}
Let $z \in C$. Suppose that $QC \subset C$. Then
\begin{equation*}
    Q (x - z) = Q x - Q z \in Q x - C \,.
\end{equation*}
So dual invariance holds. Now suppose that $C$ is $\Group$-invariant. 
Averaging implies that $Q z \preceq_{\Group} z$ and so
\begin{equation*}
    Q z \in \conv(\Group \cdot z) \subseteq C \,.
\end{equation*}
Then $Q C \subseteq C$ and dual invariance holds.
\end{proof}

\subsection{Proof of Lemma \ref{lem:cut_orbitope}}
\begin{proof}
Since both $\conv(G\cdot u)$ and 
$\set{I \circ B \given B \in \Cut_p}$ are closed
convex sets, it is enough to show that they have the same support function.
Let $d$ be the vector of diagonal entries of $D$. Then
$g \cdot U = U \circ d d^T$. So for any $Z \in \Sym_p$,
\begin{align*}
  \max_{g \in G} \innerp{g \cdot U}{Z}
  &= \max_{d \in \set{-1,+1}^m} \innerp{U \circ d d^T}{Z} \\
  &= \max_{d \in \set{-1,+1}^m} \innerp{U \circ Z}{d d^T} \\
  &= \max_{B \in \Cut_m} \innerp{U \circ Z}{B} \\
  &= \max_{B \in \Cut_m} \innerp{U \circ B}{Z} \,.
\end{align*}
Above we have used the fact that the support function of a set is the same as
that of its closed convex hull 
\citep[Proposition C.2.2.1]{HUL:fundamentals}. Thus,
$\conv(G \cdot U) = \set{U \circ B \given B \in \Cut_m}$.
\end{proof}

\subsection{Proof of Proposition \ref{pro:arcsin_cut}}
\begin{proof}
Since $\sin(\pi/2) = - \sin(-\pi/2) = 1$, it follows that $K$ contains the rank-$1$ correlation matrices $\set{x x^T \given x \in \set{-1,+1}^p}$.  Therefore,
\begin{equation*}
  \Cut_p = \conv\set{x x^T \given x \in \set{-1, +1}^p} \subseteq \conv(K) \,.
\end{equation*}
In order to show the reverse conclusion, it is enough for us to show that 
$K \subseteq \Cut_p$.
Let $Z$ and $Y$ be i.i.d. Gaussian random vectors with correlation matrix 
$\Sigma$.  For $i,j \in \bracks{p}$, it is well-known that
\begin{align*}
  &
    \E \braces{\sign(Z_i - Y_i) \sign(Z_j - Y_j)} \\
  &\quad =
    \Pr\braces{(Z_i - Y_i) (Z_j - Y_j) > 0 }
    - \Pr\braces{(Z_i - Y_i) (Z_j - Y_j) < 0 }
  \\
  &\quad =
    \frac{2}{\pi} \arcsin(\Sigma_{ij})
\end{align*}
\citep[see, e.g.,][p.~827]{kruskal:ordinal}.
Since
\begin{equation*}
  \sign(Z - Y)\sign(Z - Y)^T
  \in
  \set[\big]{ y y^T \given y \in \set{-1, +1}^p }
\end{equation*}
almost surely, it follows that
\begin{align*}
    \frac{2}{\pi} \arcsin[ \Sigma ]
    &= \E\braces{ \sign(Z - Y)\sign(Z - Y)^T } \\
    &\in \conv\parens[\big]{\set[\big]{ y y^T \given y \in \set{-1, +1}^p }}
    \\
    &= \Cut_p
    \,.\qedhere
\end{align*}
\end{proof}

\subsection{Proof of Lemma \ref{lem:ultrametric_psd}}
\begin{proof}
Suppose that the ultrametric inequality were violated so that 
\begin{equation*}
  U_{ik} < \min(U_{ij}, U_{jk})
\end{equation*}
for some $i,j,k$. Then $U_{ik} = 0$ and $U_{ij} = U_{jk} = 1$ and 
the corresponding principal submatrix 
\begin{equation*}
  \begin{bmatrix}
    1 & 1 & 0 \\
    1 & 1 & 1 \\
    0 & 1 & 1
  \end{bmatrix}
\end{equation*}
is indeterminate, so $U$ cannot be positive semidefinite. Conversely, 
suppose that the ultrametric inequality is satisfied. Then $B$ is  
ultrametric and hence positive semidefinite 
\citep[Theorem 3.5]{dellacherie.martinez.ea:inverse}.
\end{proof}
  
\printbibliography
\end{document}